\documentclass[12pt]{article}
\usepackage{diagbox}
\usepackage{amssymb}
\usepackage{amsmath,bm}
\usepackage{multirow}
\usepackage{makecell}
\usepackage{graphicx}
\usepackage{subfigure}
\usepackage{float}
\usepackage[titletoc, title]{appendix}
\usepackage{color}
\usepackage{lipsum}
\usepackage{amsfonts}
\usepackage{graphicx}
\usepackage{epstopdf}
\usepackage{algorithmic}
\usepackage{amsopn}
\usepackage[colorlinks,linkcolor=red,anchorcolor=blue,citecolor=green,CJKbookmar
ks=True]{hyperref}
\usepackage{cleveref}
\usepackage{amsthm}

\allowdisplaybreaks
\usepackage{color}
\usepackage{amssymb}
\usepackage{amsmath,bm}
\usepackage{multirow}
\usepackage{makecell}
\usepackage{graphicx}
\usepackage{subfigure}
\usepackage{float}
\usepackage{lipsum}
\usepackage{amsfonts}
\usepackage{graphicx}
\usepackage{epstopdf}
\usepackage{algorithmic}
\usepackage{amsopn}

\newcommand{\cb}[1]{{\color{blue}#1}}

\newcommand{\eq}[1]{\begin{align}#1\end{align}}
\newcommand{\eqn}[1]{\begin{align*}#1\end{align*}}
\newcommand{\h}{\mathfrak{h}}

\newtheorem{corollary}{Corollary}[section]
\newtheorem{remark}{Remark}[section]
\newtheorem{theorem}{Theorem}[section]
\newtheorem{lemma}{Lemma}[section]

\newcommand{\MP}{\mathcal{P}}

\begin{document}
	\title{A hybridizable discontinuous Galerkin method for the indefinite time-harmonic Maxwell
		equations\thanks{The research of  G. Chen is supported by National Natural Science Foundation of China (NSFC) under grant no. 12171341  and 12422115, and Opening Foundation of Agile and Intelligent Computing Key Laboratory of Sichuan
			Province.
			The research of H. Wu is supported in part by the NSF of China grants 12171238, 12261160361, and 11525103.
			The research of L. Xu is supported in part by National Natural Science Foundation of China (NSFC) under grant no.12071060 and no. 62231016.}}
	
	\date{}
	\author{
		Gang Chen%
		\thanks{School of Mathematics, Sichuan University, 610064, China. (email:{cglwdm@scu.edu.cn}).}
		\and
		Haijun Wu
		\thanks{
			Department of Mathematics, Nanjing University, Nanjing, 210093,  China
			(email:{hjw@nju.edu.cn}). }	
		\and
		Liwei Xu%
		\thanks{School of Mathematics Sciences, University of Electronic Science and Technology of China, Chengdu, 611731, China.
			The third author's research is supported in part by a Key Project of the Major Research Plan of National Natural Science Foundation of China (NSFC) grant no. 91630202 and National Natural Science Foundation of China (NSFC) grant no. 11771068. (Corresponding author: Liwei Xu, email:{xul@uestc.edu.cn}).}
	}
	\maketitle
	\begin{abstract} 
		In this paper, we aim to develop a hybridizable discontinuous Galerkin (HDG) method for the indefinite time-harmonic Maxwell equations with the perfectly conducting boundary in the three-dimensional space. First, we derive the wavenumber explicit regularity result, which plays an important role in the error analysis for the HDG method.
		Second,  we prove a discrete inf-sup condition which holds for all positive mesh size $h$, for all wavenumber $k$, and for general domain $\Omega$.
		Then, we establish the optimal order error estimates of the underlying HDG method with constant independent of the wavenumber. The theoretical results are confirmed by numerical experiments.
		
		keywords: Maxwell equations, HDG method, low regularity
	\end{abstract}

	\section{Introduction}
	Let $\Omega$ be a bounded simply-connected Lipschitz polyhedron in $\mathbb{R}^3$ with a connected boundary $\Gamma:=\partial\Omega$. We consider the following case of the time-harmonic Maxwell equations with the perfectly conducting boundary condition in a mixed form \cite{MR1929626}:
	\par\smallskip\noindent
	Find the electric field $\bm u$ and the Lagrange multiplier $p$ such that
	\begin{subequations}\label{source}
		\begin{align}
			\bm\nabla\times \bm\nabla\times \bm{u}-{k}^2\bm u+({k}^2+1)\nabla p&=\bm{f}&\text{ in }\Omega, \label{o1} \\
			\nabla\cdot\bm{u}&=0&\text{ in }\Omega,\label{o2}  \\
			\bm{n} \times \bm{u} &=\bm{0} &\text{ on }\Gamma,\label{boundary}\\
			p&=0 &\text{ on }\Gamma.\label{o4}
		\end{align}
	\end{subequations}
	Here, $\bm{n} $ is the outward unit normal vector to the boundary $\Gamma$, $\bm{f}\in  [L^2(\Omega)]^3$ is a given external source  field, ${k}:=\omega\sqrt{\varepsilon_0\mu_0}$ is a real wavenumber, where $\omega>0$ is a given temporal frequency, and $\varepsilon_0$ and $\mu_0$ are the electric permittivity and the magnetic permeability of the free space, respectively.
	Note that in the special case here (relative electric permittivity of the medium equals one and perfect conducting boundary condition), the real and imaginary parts are decoupled, and thus we assume that $\bm u$, $p$ and $\bm f$ are real.

	The numerical solution of the indefinite time-harmonic Maxwell equations suffers from the following two challenges. First, on a non-convex domain, the solution of Maxwell equations is only in $[H^s(\Omega)]^3$ with $s\in (1/2,1)$. A direct application of continuous finite element methods will result in a discrete solution that converges to a spurious solution.
	Second, the quality of numerical solutions to the Maxwell equation depends significantly on
	the wavenumber $k$. Different methods are applied to solve the electromagnetic models, including boundary integral methods \cite{MR3207533,MR3267105,MR1822275}, boundary element methods \cite{MR744924,MR1935806},  and finite element methods.
	The finite element method was the most popular computational technique for solving the time-harmonic Maxwell equation. In particular, finite element methods using $\bm{H}(\text{curl};\Omega)$-conforming edge elements have
	been studied in vast literatures for \eqref{source} and its reduced problem where $\nabla\cdot \bm f=0$, see  \cite{MR592160,MR864305,MR2009375,MR2059447,MR2536902,MR2869030,Brenner2006DivF}. Meanwhile, preconditioners for finite element methods solving the {\em indefinite} Maxwell equations were investigated in  \cite{MR1609607,MR2902419,MR1933811,MR2310392} and the references therein. Since the late 1970s, the discontinuous Galerkin (DG) methods have become increasingly popular due to some attractive features, including preserving local conservation of physical quantities, flexible in meshing and parallel computation, ease of design, implementation, and $hp$-adaptive strategy. DG methods for solving the time-harmonic Maxwell equations with zero wavenumber were first developed in \cite{MR1972732, MR2051073}, and then interior penalty discontinuous Galerkin (IPDG) methods for the {\em indefinite} Maxwell equations were studied in \cite{MR1929626,MR2671288} as well as for curl-curl problems in \cite{Brenner2008IPCurl}. We notice that the constants in the stability results and error estimates of the IPDG methods in \cite{MR1929626} are highly dependent on the wavenumber. In \cite{MR3265182,MR2983024}, the authors proposed and analyzed DG methods for the {\em indefinite} Maxwell equations with the {\em impedance boundary condition}, and derived the wavenumber explicit convergence results.
	We would like to remark that
	there are no research on the error estimates with explicit wavenumber dependence for the {\em indefinite} Maxwell equations with the {\em perfect conducting boundary}.
	We should also mention that in \cite{MR2263045,MR2324460,MR2511729,MR2220916}, the DG methods for the
	spurious Maxwell modes were considered.

	In recent years, the hybridizable discontinuous Galerkin (HDG) method, a type of DG methods, has been successfully applied to solve various types of differential equations, see \cite{MR2629996,MR3044180,MR2833489,MR3168286,MR3342199,MR3626530,MR3813574,MR4036984,MR4081917,MR3954449,MR3874791} and references therein. The HDG method retains the advantages of standard DG methods and can significantly reduce the number of degrees of freedom, leading to a substantial reduction in the computational cost. The first work \cite{MR2822937} that applied the HDG methods to solve the {\em indefinite} time-harmonic Maxwell equations appeared in 2011. In this paper, two HDG schemes were introduced and numerical results were reported to illustrate the performance of the proposed schemes, and however the convergence analysis was not given yet.
	In addition to this pioneer work on Maxwell equations using the HDG methods, many other HDG methods have been developed and analyzed in past few years. The HDG methods for the time-harmonic Maxwell equations with the zero wavenumber have been proposed and analyzed in \cite{MR3608330,MR3771897,Chen-2018-01}, and the HDG methods for the indefinite time-harmonic Maxwell equations with the impedance boundary condition have been designed in \cite{MR3518362,MR3626528} where the error estimates with the constants being dependent explicitly on the wavenumber have been established with the help of the regularity results in \cite{MR3265182,Maxwell-IBC}.

	%
	
	In this paper, we propose a new HDG method for the {\em indefinite} time-harmonic Maxwell equations \eqref{source} with the {\em perfectly conducting boundary condition}. We first derive the wave-number explicit regularity result of the Maxwell equations, i.e., there exists a regularity index $s\in (1/2,1]$ depending on $\Omega$, such that $\bm{u}\in \bm H^s({\rm curl};\Omega)$ and
	\begin{align*}
		({k}+1)\|\bm{u}\|_{s}+\|\bm\nabla\times\bm{u}\|_{s}\lesssim \mathcal M_k\|\bm{f}\|_0,
	\end{align*}
	where $\mathcal M_k$ is defined as
	\begin{align*}
		\mathcal M_k=\sup_{\lambda_i\in E_{\lambda}}\left|\frac{\lambda_i+k^2}{\lambda_i-k^2}\right|,
	\end{align*}
	and $E_{\lambda}$ is the set of all eigenvalues of the corresponding eigenvalue problems.
	A similar regularity result has been reported in \cite{MR4423464}, where $\Omega$ is supposed to be smooth, and therefore $s=1$. Based on such regularity result, we establish the error estimates for the proposed HDG method as follows,
	\begin{align*}
		\|\bm{u}_h-\bm{u}\|_0&\lesssim  ({\mathcal M_k}  h^{s^*}+{\mathcal M}_k^2h^2) \|\bm{f}\|_0,
	\end{align*}
	given that $\Omega$ is convex and $kh^s + \mathcal  M_k k^2h^{s^*}\le C_0$, where $\bm u_h$ is the approximation of $\bm u$, $s\in (\frac 1 2 ,1 ]$ and $s^{*}\in(1,2]$ depend on $\Omega$, and $C_0$ is a constant independent of the wavenumber. To the author's knowledge, such convergence result is also the {\em first} of its kind in the numerical study HDG methods of the {\em indefinite} time-harmonic Maxwell equations.

	The rest of this paper is organized as follows. In \Cref{section2}, we present a regularity result of the {\em indefinite} time-harmonic Maxwell equations. In \Cref{S:HDG,section4}, we propose a new HDG method and establish its well-posedness. In \Cref{section5}, we develop the convergence analysis of the HDG method based on the regularity and stability results. In \Cref{section6}, numerical experiments are performed to verify the theoretical results.

	Throughout this paper, we use $C$ to denote a positive constant independent of mesh size and the wavenumber $k$, not necessarily the same at its each occurrence.
	{For convenience we use the shorthand notation $a\lesssim b$ and $a \gtrsim b$ for the inequality $ a\le Cb$ and $ b\le Ca$. $a\backsimeq b$ stands for $a\lesssim b$ and $a\gtrsim b$.}
	%
	\section{The wavenumber explicit regularity}
	\label{section2}
	For any bounded Lipschitz domain $\Lambda\subset \mathbb{R}^d$ $(d=2,3)$, let $H^{m}(\Lambda)$ and $H^m_0(\Lambda)$  denote the usual  $m^{th}$-order Sobolev spaces on $\Lambda$, and $\|\cdot\|_{m, \Lambda}$, $|\cdot|_{m,\Lambda}$  denote the norm and semi-norm on these spaces. We use $(\cdot,\cdot)_{m,\Lambda}$ to denote the inner product of $H^m(\Lambda)$, with $(\cdot,\cdot)_{\Lambda}:=(\cdot,\cdot)_{0,\Lambda}$.   We use $v|_{\partial\Lambda}$ to denote the trace of $v$ on $\partial\Lambda$.
	When $\Lambda=\Omega$, we abbreviate $\|\cdot\|_{m }:=\|\cdot\|_{m, \Omega}$, $|\cdot|_{m}:=|\cdot|_{m,\Omega}$, $(\cdot,\cdot):=(\cdot,\cdot)_{\Omega}$.
	In particular,  for a surface $F$ we use $\langle\cdot,\cdot\rangle_{F}$
	to denote the $L^2$ inner products on  $F$.
The bold face fonts will be used for vector (or tensor) analogues of the Sobolev spaces along with vector-valued (or tensor-valued) functions.
We first define the function spaces as below
\begin{align*}
	\bm{H}(\text{curl};\Omega)&:=\{\bm{v}\in [L^2(\Omega)]^3: \bm\nabla\times\bm{v}\in [L^2(\Omega)]^3 \},\\
	\bm{H}^s(\text{curl};\Omega)&:=\{\bm{v}\in [H^s(\Omega)]^3: \bm\nabla\times\bm{v}\in [H^s(\Omega)]^3 \}\text{ with }s\ge 0,\\
	\bm{H}_0(\text{curl};\Omega)&:=\{\bm{v}\in \bm{H}(\text{curl};\Omega): \bm{n} \times\bm{v}|_{\Gamma}=\bm{0} \},\\
	\bm{H}(\text{div} ;\Omega)&:=\{\bm{v}\in [L^2(\Omega)]^3: \nabla\cdot\bm{v}\in L^2(\Omega) \},\\
	\bm{H}_0(\text{div} ;\Omega)&:=\{\bm{v}\in \bm{H}(\text{div} ;\Omega): \bm{n}\cdot\bm{v}|_{\Gamma}=0 \},\\
	\bm{H}(\text{div} ^0;\Omega)&:=\{\bm{v}\in \bm{H}(\text{div} ;\Omega): \nabla\cdot\bm{v}=0\},
\end{align*}
and
\begin{align*}
	\bm{X} &:=\bm{H}(\text{curl};\Omega)\cap \bm{H}(\text{div} ;\Omega),
	&\bm{X} _N:=\bm{H}_0(\text{curl};\Omega)\cap \bm{H}(\text{div} ;\Omega),\\
	\bm{X} _{N,0}&:=\bm{H}_0(\text{curl};\Omega)\cap \bm{H}(\text{div}^0 ;\Omega),
	&\bm{X} _T:=\bm{H}(\text{curl};\Omega)\cap \bm{H}_0(\text{div} ;\Omega).
\end{align*}
%
%
\begin{lemma}[Helmholtz decomposition \cite{MR2059447}] \label{hel}  {For any $\bm{v}\in\bm L^2(\Omega)$,} there exist functions $\bm z\in \bm H^1(\Omega)$ and $\psi\in H^1_0(\Omega)$ such that
\begin{align}\label{dec}
	\bm{v}= {\bm{z}+\nabla\psi,\quad\nabla\cdot \bm z=0, }
\end{align}
and
\begin{align}\label{decc}
	\|\bm{z}\|_{0}+\|\nabla\psi\|_0\lesssim \|\bm{v}\|_0.
\end{align}
\end{lemma}
We define the bilinear form as
\begin{align}\label{eq:apm}
a^{\pm}(\bm{u},\bm{v})=(\bm\nabla\times\bm{u},\bm\nabla\times\bm{v})\pm{k}^2(\bm{u},\bm{v}).
\end{align}
By testing the first equation of \eqref{source} with functions $\bm{v}\in\bm{X} _{N,0}$, it is easy to check that the solution $\bm{u}$ of \eqref{source} is also the solution of the weak problem: Find $\bm{u}\in \bm{X} _{N,0}$ such that
\begin{align}
a^{-}(\bm{u},\bm{v})=( {\bm{f}},\bm{v})\qquad\forall\bm{v}\in\bm{X} _{N,0}. \label{s-1}
\end{align}
Similarly, by testing the first equation of \eqref{source} with $\nabla q$ where $q\in H^1_0(\Omega)$, we observe that the solution $p$ of \eqref{source} is also the solution of the weak problem:
Find $p\in H^1_0(\Omega)$ such that
\begin{align}
({k}^2+1)(\nabla p, \nabla q)=( {\bm{f}},\nabla q )\qquad\forall q\in H^1_0(\Omega). \label{s-2}
\end{align}

Now, we introduce the following auxiliary problem: Find $\widetilde{\bm{u}}\in \bm{X} _{N,0}$ such that
\begin{align}
a^{+}(\widetilde{\bm{u}},\bm{v})=( {\bm{f}},\bm{v})\qquad \forall\bm{v}\in\bm{X} _{N,0}. \label{s-1-1}
\end{align}
{Define} the solution operator $\bm{\mathcal K}_k: [L^2(\Omega)]^3\mapsto  \bm{X} _{N,0}$ as:
for any  $\bm{w}\in [L^2(\Omega)]^3$, find
$\bm{\mathcal K}_k\bm{w}\in\bm{X} _{N,0}$  such that
\begin{align}\label{s-1-2}
a^{+}(\bm{\mathcal K}_k\bm{w},\bm{v})=-2{k}^2(\bm{w},\bm{v})\qquad\forall\bm{v}\in\bm{X} _{N,0},
\end{align}
{and let} $\bm u\in \bm X_{N,0}$ be the weak solution to \eqref{source} (i.e., the solution to \eqref{s-1}), then it is obvious that
\begin{align*}
a^{+}((\bm{\mathcal I}+\bm{\mathcal K}_k)\bm{u},\bm{v})=a^{+}(\widetilde{\bm{u}},\bm{v}),
\end{align*}
which leads to the following relation:
\begin{align}
(\bm{\mathcal I}+\bm{\mathcal K}_k)\bm{u}=\widetilde{\bm{u}}. \label{op}
\end{align}

We recall the classical estimation for vector potential $\bm{v}\in \bm{X}$ in the following lemma.
\begin{lemma}[cf. {\cite[Proposition 7.4]{Maxwell-1997}}]\label{conceive} For any
$\bm{v}\in \bm{X}_{N}$ or $\bm v\in \bm X_T$, there holds:
\begin{align*}
	\|\bm{v}\|_0\lesssim \|\bm\nabla\times\bm{v}\|_0+\|\nabla\cdot\bm{v}\|_0.
\end{align*}
\end{lemma}

The well-posedness of problems \eqref{s-1-1} and  \eqref{s-1-2}  are established in the next two lemmas.
\begin{lemma}\label{lemma23} The problem \eqref{s-1-1} has a unique solution  {satisfying the following estimate}:
\begin{align}\label{stablity-tilde}
	({k}^2+1)\|\widetilde{\bm{u}}\|_0+({k}+1)\|\bm\nabla\times\widetilde{\bm{u}}\|_0\lesssim  \|\bm{f}\|_0.
\end{align}
\end{lemma}
\begin{proof}
By \Cref{conceive}, we have that
$\|\bm{v}\|_0\lesssim\|\bm\nabla\times\bm{v}\|_0$ for any $\bm{v}\in\bm{X} _{N,0}$. Therefore the bilinear form $a^+$ is continuous and coercive under the norm $\big(({k}^2+1)\|\bm{v}\|^2_0+\|\bm\nabla\times\bm{v}\|^2_0\big)^\frac12$. By the Lax-Milgram lemma,
\eqref{s-1-1} attains a unique solution $\widetilde{\bm{u}}$, and there holds
\begin{align*}
	({k}^2+1)\|\widetilde{\bm{u}}\|^2_0+\|\bm\nabla\times\widetilde{\bm{u}}\|^2_0\le C  \|\bm{f}\|_0\|\widetilde{\bm{u}}\|_0\le {\frac12({k}^2+1)\|\widetilde{\bm{u}}\|^2_0+C({k}^2+1)^{-1}\|\bm{f}\|_0^2},
\end{align*}
which implies \eqref{stablity-tilde}.
\end{proof}
\begin{lemma}\label{lemma24}  There hold that
\begin{itemize}
	\item[{\rm(i)}] for any given $\bm{w}\in [L^2(\Omega)]^3$, the problem \eqref{s-1-2} has a unique solution $\bm{\mathcal K}_k\bm{w}$ satisfying the following stability estimate:	
\end{itemize}
\begin{align}
	(k+1)\|\bm{\mathcal K}_k\bm{w}\|_0+\|\bm\nabla\times\bm{\mathcal K}_k\bm{w}\|_0
	\lesssim {k}\|\bm{w}\|_0;\label{stablity-K}
\end{align}	
\begin{itemize}
	\item [{\rm(ii)}] $\bm{\mathcal K}_k$ is a self-adjoint and  compact operator on $[L^2(\Omega)]^3$;
	\item [{\rm(iii)}] $\bm X_{N,0}$ admits a countably infinite orthonormal basis $\{\bm u_i\}$ of eigenvectors of $\bm{\mathcal K}_k$, with corresponding eigenvalues $\{\mu_i\}\subset\mathbb R$ satisfying $\mu_i\to 0$.
\end{itemize}
\end{lemma}
\begin{proof} It is clear that\\
(i) is a consequence of  \Cref{lemma23} with $\bm f=-2k^2\bm w$;\\
(ii) follows directly from  the definition \eqref{s-1-2} of $\bm{\mathcal K}_k$  and  the compact embedding of $ \bm{X}_{ N,0}$  in  $[L^2(\Omega)]^3$   (cf. \cite[Page 87, Theorem 4.7]{MR2059447});\\
(iii) follows from (ii), the spectral theory of compact self-adjoint  operator on Hilbert space   (cf. \cite[Page 60, Theorem 6.21]{MR3027462}), and the fact that  the orthogonal complement of the kernel of $\bm{\mathcal K}_k$ is $\bm X_{N,0}$ , which {could} be proved by  the definition \eqref{s-1-2} of $\bm{\mathcal K}_k$.
\end{proof}

Let $\{\lambda_{i}\}_{i=1,2,\cdots}$ be the set of nonzero eigenvalues of the Maxwell operator $\bm\nabla\times\bm\nabla\times$ on $H_0({\rm curl};\Omega)$ such that $0<\lambda_1\le\lambda_2\le\cdots $
and $\bm u_i\in\bm H_0({\rm curl};\Omega)$ be the corresponding eigenfunctions {satisfying}
\begin{subequations}\label{eigen}
\begin{align}
	(\bm\nabla\times\bm u_i, \bm\nabla\times\bm v)&=\lambda_i ({\bm{u}}_i,\bm v)\qquad \forall \bm v\in \bm H_0({\rm curl};\Omega),\\
	\| \bm u_i\|_0&=1.
\end{align}
\end{subequations}

\begin{lemma} The eigenvalues of $\bm{\mathcal K}_k$ {consist} of $\mu_i:=\dfrac{-2k^2}{\lambda_i+k^2}, i=1,2,\cdots$, with corresponding eigenfunctions $\bm u_i$.\label{lemma2.5eigval}
\end{lemma}
\begin{proof} First we note that $\bm u_i\in \bm X_{N,0}$ since $\lambda_i\neq 0$. It follows from the Helmholtz decomposition \Cref{hel} that  \eqref{eigen} is equivalent to the following eigenvalue problem: Find $(\lambda_i,\bm u_i)\in \mathbb R \times\bm X_{N,0}$ such that
\begin{subequations}\label{eigen2}
	\begin{align}
		(\bm\nabla\times\bm u_i, \bm\nabla\times\bm v)&=\lambda_i ({\bm{u}}_i,\bm v)\qquad \forall\, \bm v\in \bm X_{N,0},\label{eigen2a}\\
		\| \bm u_i\|_0&=1.
	\end{align}
\end{subequations}
Clearly, \eqref{eigen2a} is equivalent to
\eqn{a^{+}({\bm{u}}_i,\bm v)=(\lambda_i+k^2)({\bm{u}}_i,\bm v)\qquad \forall \,\bm v\in \bm X_{N,0},}
and we complete the proof by using the definition of $\bm{\mathcal K}_k$ and some simple calculations.

\end{proof}

The well-posedness of \eqref{s-1} is given in the next lemma.
\begin{lemma}\label{LeigKk}
Suppose ${k}^2$ is not an eigenvalue of \eqref{eigen}, then the problem \eqref{s-1} has a unique solution. Moreover, the inverse of $\bm{\mathcal I}+\bm{\mathcal K}_k$ exists, and
\begin{align}\label{KC}
	\|(\bm{\mathcal I}+\bm{\mathcal K}_k)^{-1}\bm w\|_0\le \mathcal M_k\|\bm w\|_0\qquad \forall \,\bm w\in \bm X_{N,0},
\end{align}
where
\eq{\label{Mk}
	\mathcal M_k:=\sup_{i=1,2,\cdots}\left|\frac{\lambda_i+k^2}{\lambda_i-k^2}\right|\;.}
	\end{lemma}
	\begin{proof}    From \Cref{lemma2.5eigval}, the eigenvalues of $\bm{\mathcal I}+\bm{\mathcal K}_k$ are given by $\dfrac{\lambda_i-k^2}{\lambda_i+k^2}, i=1,2,\cdots,$ which are all nonzero. Therefore, $\bm{\mathcal I}+\bm{\mathcal K}_k$ is invertible and \eqref{KC} follows from the combination of \Cref{lemma24} (iii) and the $\bm L^2$-orthogonality of the basis $\{\bm u_i, i=1,2,\cdots\}.$  Finally, the well-posedness of \eqref{s-1} follows by using  \eqref{op}.
	\end{proof}

	\begin{remark} 
Concerning the feature of $\mathcal M_k$, it could be arbitrarily large if $k^2$ approaches to any nonzero Maxwell eigenvalue. In addition, we could derive a lower bound for $\mathcal M_k$ as considering that $\Omega$ is a convex polyhedron. Similar to \cite[Theorem 4.1]{MR1672271}, the nonzero Maxwell eigenvalues are also eigenvalues of the Laplace operator with Neumann boundary condition, whose $n^{\rm th}$ eigenvalue $\hat\lambda_n$  behaves asymptotically as  $\hat\lambda_n\sim \hat c n^\frac23$, where $\hat c$ is a constant depending only on the domain $\Omega$ (see e.g.\cite{Clark1967,Zhang2015}). Therefore, if the wave number $k$ is sufficiently large and   $k^2$ is located in $(\hat\lambda_n, \hat\lambda_{n+1})$ for some large $n$, then it holds that
\eqn{\mathcal M_k=\sup_{i=n,n+1}\left|\frac{\lambda_i+k^2}{\lambda_i-k^2}\right|\gtrsim \frac{n^\frac23}{(n+1)^\frac23-n^\frac23}
	\gtrsim n\gtrsim k^3.}
	\end{remark}
	
	In the rest of  this section, we derive  stability and  regularity results for the {\em indefinite} time-harmonic Maxwell's equations \eqref{source}.

	\begin{lemma}\eqref{source}  has a unique weak solution $(\bm{u},p)$,  which takes the following  stability estimate
\begin{align}\label{eq:regest}
	({k}^2+1)\|\bm{u}\|_{0}+({k}+1)\|\bm\nabla\times\bm{u}\|_{0}\lesssim {\mathcal M_k}\|\bm{f}\|_0.
\end{align}
\end{lemma}
\begin{proof}
By combining \eqref{op}, \eqref{stablity-tilde}, \eqref{KC}, we get
\begin{equation}\label{eq:218}
	({k}^2+1)\|\bm{u}\|_{0}=({k}^2+1)\|(\bm{\mathcal I}+\bm{\mathcal K}_k)^{-1}\widetilde{\bm{u}}\|_{0}
	\le{\mathcal M_k}({k}^2+1)\|\widetilde{\bm{u}}\|_{0}
	\lesssim {\mathcal M_k}\|\bm{f}\|_0.
\end{equation}
It follows from \eqref{op}, \eqref{stablity-tilde}, \eqref{stablity-K} and \eqref{eq:218} that
\begin{align}\label{eq:219}
	\begin{split}
		({k}+1)\|\bm\nabla\times\bm{u}\|_{0}&\le ({k}+1)\|\bm\nabla\times\bm{\mathcal K}_k\bm{u}\|_{0}+({k}+1)\|\bm\nabla\times \widetilde{\bm{u}}\|_{0} \\
		&\lesssim ({k}^2+1)\|\bm{u}\|_0+\|\bm{f}\|_0\\
		&\lesssim {\mathcal M_k}\|\bm{f}\|_0,
	\end{split}
\end{align}
which together with \eqref{eq:218}  implies that \eqref{eq:regest} holds.	
\end{proof}

\begin{lemma}[cf. {\cite[Proposition 3.7]{MR1626990}}]  \label{embed}
If the domain $\Omega$ is a Lipschitz polyhedron, then $\bm{X}_T(\Omega)$ and $\bm{X}_N(\Omega)$ are continuously embedded into $[H^s(\Omega)]^3$ for some  real number $s\in (1/2,1].$
\end{lemma}

To end this section, we present the wavenumber explicit regularity result of \eqref{source}.
\begin{theorem} [Regularity]\label{reg}
Let $(\bm{u},p)$ be the solution of \eqref{source}, then there exists a regularity index $s\in (1/2,1]$ depending on $\Omega$, such that $\bm{u}\in \bm H^s({\rm curl};\Omega)$ and
\begin{subequations}
	\begin{align}\label{eq:FreqExpReg}
		({k}+1)\|\bm{u}\|_{s}+\|\bm\nabla\times\bm{u}\|_{s}\lesssim    {\mathcal M_k} \|\bm{f}\|_0.
	\end{align}
	Moreover, if $\bm f\in \bm H({\rm div};\Omega)$, it holds  {that $p\in H^{s+1}(\Omega)$ and}
	\begin{align}\label{eq:FreqExpReg2}
		({k}^2+1)\|p\|_{1+s}\lesssim  \|\nabla\cdot\bm f\|_0.
	\end{align}
	In particular, $s=1$ if $\Omega$ is convex. Furthermore, if $\bm f\in \bm H({\rm div^0};\Omega)$ and $\Omega$ is convex, there exists  some regularity index $s^*\in (1,2]$ depending on $\Omega$, such that $\bm{u}\in \bm H^{s^*}(\Omega)${  , and}
	\begin{align}\label{rc}
		\|\bm u\|_{s^*}\lesssim {\mathcal M_k} \|\bm{f}\|_0.
	\end{align}
\end{subequations}
\end{theorem}
\begin{proof}
Let $\bm{u}$ be the solution of \eqref{source}. Note that on the boundary $\Gamma$, $\bm{n}\cdot(\bm\nabla\times\bm{u})=0$ since $\bm{n}\times\bm{u}=\bm{0}$.
Hence, there exists a real number  $s_1\in (1/2,1]$ in terms of \Cref{embed} such that
\begin{align*}
	\|\bm{u}\|_{s_1}&\lesssim  \|\bm{u}\|_0+\|\bm\nabla\times\bm{u}\|_0,\\
	\|\bm\nabla\times\bm{u}\|_{s_1}&\lesssim   \|\bm\nabla\times\bm{u}\|_0+\| \bm\nabla\times(\bm\nabla\times\bm{u})\|_0.
\end{align*}
We apply $\nabla\cdot$ to \eqref{o1}, and then combine \eqref{o2} to obtain
\begin{subequations} \label{ellip}
	\begin{align}
		(k^2+1)\Delta p&=\nabla\cdot\bm f,\qquad\text{in }\Omega,\\
		p&=0,\qquad\quad\text{on }\partial\Omega.
	\end{align}
\end{subequations}
Clearly, $(k^2+1)\|\nabla p\|_0\lesssim \|f\|_0$.
Since $\Omega$ is a  Lipschitz polyhedron, by the standard elliptic regularity results  \cite{MR961439}, we obtian the regularity result for \eqref{ellip}: there exists a real number $s_2\in (1/2,1]$ such that
\begin{align*}
	\|({k}^2+1)p\|_{1+s_2}\lesssim  \|\nabla\cdot\bm{f}\|_0.
\end{align*}
Therefore, inequalities \eqref{eq:FreqExpReg} and \eqref{eq:FreqExpReg2} hold with  $s=\min(s_1,s_2)$. The last inequality could be derived by using the regularity result in \cite[\S4]{MR1753704} and \eqref{eq:regest}. This completes the proof.
\end{proof}
\begin{remark}
In \cite{MR1929626}, it has been proved that
\begin{align*}
	\|\bm{u}\|_{s}+\|\bm\nabla\times\bm{u}\|_{s}\le C_{\rm reg}   \|\bm{f}\|_0,
\end{align*}
where $C_{\rm reg}$ depends on $k$. Here, we  show explicitly
how the constant $C_{\rm reg}$  is dependent on the wavenumber $k$
\end{remark}


\section{An HDG method}\label{S:HDG}
By introducing $\bm{r}=\bm\nabla\times\bm{u}$, we can rewrite \eqref{source} as: Find $(\bm{r},\bm{u},p)$ such that
\begin{subequations}\label{mixed}
\begin{align}
	\bm{r}-\bm\nabla\times\bm{u}&=\bm{0}&\text{ in }\Omega,\\
	\bm\nabla\times\bm{r}-{k}^2\bm{u}+({k}^2+1)\nabla p&=\bm{f}& \text{ in }\Omega, \label{mix0source}\\
	\nabla\cdot\bm{u}&=0&  \text{ in }\Omega,  \\
	\bm{n}\times \bm{u} &=\bm{0}&\text{ on }\Gamma,\\
	p &=0& \text{ on } \Gamma.
\end{align}
\end{subequations}

Let $\mathcal{T}_h=\bigcup\{T\}$  be a shape-regular partition of the domain $\Omega$ consisting of simplexes. For any $T\in\mathcal{T}_h$, let $h_T$ be the infimum of the diameters of spheres containing $T$ and denote the mesh size $h:=\max_{T\in\mathcal{T}_h}h_T$. Let $\mathcal{F}_h=\bigcup\{F\}$ be the union of all faces of $T\in\mathcal{T}_h$, and $\mathcal{F}_h^I$ and $\mathcal{F}_h^B$ be the set of interior faces and boundary faces, respectively. We denote by $h_F$ the smallest diameter of all circles containing face $F$.
Moreover, we define the mesh-size function $\h$ as
\eq{\label{h}
\h(\bm x):=\begin{cases}
	h_F, &\text{for $\bm x$ {  on faces} }  F\in \mathcal F_h,\\
	h_T, &\text{for $\bm x$ in the interior of }  T\in \mathcal T_h.
	\end{cases}}
	For any  $T\in\mathcal{T}_h$, we denote by $ \bm{n}_T $ the unit outward  normal vector to $\partial T$. We extend the definition of $\bm n$ to the boundary of elements by letting $\bm n_{\partial T}=\bm n_T$. Note that $\bm n$ is double valued on  interior faces with opposite directions. For   an interior face $F=\partial T \cap \partial T'\in \mathcal{F}_h^I$ shared by an element $T$ and an element
	$T'$, and a  piecewise  function $\bm{\phi}$,  we
	define the jump of $\bm{\phi}$ on $F$ as
	\begin{align*}
[\![ \bm{\phi}]\!]|_F:=\bm{\phi}|_T-\bm{\phi}|_{T'}.
\end{align*}
On a boundary face $F\subset{  T}\cap \partial\Omega$, we set $[\![ \bm{\phi}]\!]|_F:=\bm{\phi}$.
For $u,v\in L^2(\partial\mathcal{T}_h)$, we define the inner product and norm as
\begin{align*}
\langle u,v \rangle_{\partial\mathcal{T}_h}&=\sum_{T\in\mathcal{T}_h}\langle u,v\rangle_{\partial T},
&
\|v\|^2_{0,\partial\mathcal{T}_h}=&\sum_{T\in\mathcal{T}_h}\|v\|^2_{0,\partial T}.
\end{align*}
Discrete curl, divergent and gradient operators with respect to
mesh partition $\mathcal{T}_h$ are donated by ${\bm\nabla_h\times}$, $\nabla_h\cdot$ and $\nabla_h$, respectively.

For an integer $\ell\ge 0$, $\mathbb{P}_{\ell}(\Lambda)$  denotes the set of all polynomials defined on $\Lambda$ with degree no greater than $\ell$. For any integer ${\ell}\ge 1$ and $m\in\{{\ell}-1,{\ell}\}$, we introduce  the following  finite dimensional spaces:
\begin{align*}
\bm{R}_{h}&:=\{\bm s_h\in [L^2(\Omega)]^3:\;\bm s_h {|_T}\in [\mathbb {P}_m(T)]^3,\;\forall\, T\in\mathcal{T}_h\}, \\
\bm{U}_{h}&:=\{\bm v_h\in [L^2(\Omega)]^3:\;\bm v_h|_T\in [\mathbb {P}_{\ell}(T)]^3,\;\forall\, T\in\mathcal{T}_h\}, \\
\widehat{\bm{U}}_{h}&:=\{\widehat{\bm{v}}_h\in [L^2(\mathcal{F}_h)]^3:\;\widehat{\bm v}_h|_ {F}\in[\mathbb P_{\ell}(F)]^3,\;\forall F\in\mathcal{F}_h,\; \widehat{\bm{v}}_h\cdot\bm{n}|_{\mathcal{F}_h}=0,\;
\bm{n}\times\widehat{\bm{v}}_h|_{\Gamma}=\bm 0  \} ,\\
{M_h}&:=\{ q_h\in L^2(\Omega):\;q_h|_T\in \mathbb P_{\ell+1}(T),\;\forall\, T\in\mathcal{T}_h\},\\
{\widehat{M}_{h}}&:=\{\widehat{q}_h\in L^2(\mathcal{F}_h):\;\widehat{q}_h|_F\in \mathbb P_{\ell+1}(F),\;\forall F\in\mathcal{F}_h,\;\widehat{q}_h|_{\Gamma}=0\}.
\end{align*}

The HDG method for \eqref{source} reads: Find an approximation $(\bm r_h,\bm u_h,$ $\widehat{\bm u}_h,p_h,\widehat p_h) \in \bm{R}_h\times\bm{U}_h\times\widehat{\bm{U}}_h\times {M}_h\times \widehat{{M}}_h$ such that
\begin{subequations}\label{oror}
\begin{align}
	(\bm r_h,\bm s_h)-(\bm u_h,\bm\nabla_h\times\bm s_h)-\langle
	\bm n\times\widehat{\bm u}_h,\bm s_h \rangle_{\partial\mathcal{T}_h}&=0, \label{orora}\\
	(\bm r_h,\bm\nabla_h\times\bm v_h)+\langle\widehat{\bm n\times \bm r_h},\bm v_h \rangle_{\partial\mathcal{T}_h}
	-(k^2+1)(p_h,\nabla_h\cdot\bm v_h)
	&\label{ororb}\\
	+(k^2+1)\langle\widehat p_h,\bm n\cdot\bm v_h \rangle_{\partial\mathcal{T}_h}-k^2(\bm u_h,\bm v_h)&=(\bm f,\bm v_h),\nonumber \\
	-(k^2+1)(\bm u_h, \nabla_h q_h)
	+(k^2+1)\langle\widehat{\bm n\cdot \bm u_h},q_h \rangle_{\partial\mathcal{T}_h}&=0, \label{ororc}\\
	\langle\widehat{\bm n\times \bm r_h},\widehat{\bm v}_h \rangle_{\partial\mathcal{T}_h}&=0, \label{orord}\\
	\langle\widehat{\bm n\cdot \bm u_h},\widehat{q}_h \rangle_{\partial\mathcal{T}_h}&=0 \label{orore}
\end{align}
for any $(\bm s_h,\bm v_h,$ $\widehat{\bm v}_h,q_h,\widehat q_h) \in \bm{R}_h\times\bm{U}_h\times\widehat{\bm{U}}_h\times {M}_h\times \widehat{{M}}_h$,
where the numerical fluxes are defined as
\end{subequations}
\begin{subequations}
\begin{align}
	\widehat{\bm n\times\bm r_h}&=\bm n\times\bm r_h+ {\h^{-1}\bm n\times(\bm u_h-\widehat{\bm u}_h)\times \bm n, \quad\text{ on } \partial T, \;\forall~T \in \mathcal T_h}, \label{f1}\\
	\widehat{\bm n\cdot\bm u_h}&=\bm n\cdot\bm u_h+ {\h^{-1}(p_h-\widehat p_h), \quad \text{ on } \partial T,\; \forall~T \in \mathcal T_h}.\label{f2}
\end{align}
\end{subequations}

\begin{remark} The HDG method presented above is different from the method in \cite{MR2822937} in two aspects: the stabilization parameters  in \cite{MR2822937} are $O(1)$ and the stabilization parameters here are $O(\h^{-1})$;
the scheme in \cite{MR2822937} used $\ell^{\rm th}$ polynomials for all variables and   our method allows $({\ell}-1)^{\rm th}$ polynomials for the approximation of $\bm r$ and  $({\ell}+1)^{\rm th}$ polynomials for the approximation of $p$.	
\end{remark}

Applying \eqref{f1}--\eqref{f2} and \eqref{orord}--\eqref{orore} to eliminate $\widehat{\bm n\times\bm r_h}$ and $\widehat{\bm n\cdot\bm u_h}$ in \eqref{orora}--\eqref{ororc}  and  performing integration by parts, we get the following saddle point system:
\par\smallskip\noindent
Find $(\bm{r}_h,\bm{u}_h,\widehat{\bm{u}}_h,p_h,\widehat{p}_h) \in \bm{R}_h\times\bm{U}_h\times\widehat{\bm{U}}_h\times {M}_h\times \widehat{{M}}_h$ such that
\begin{subequations}\label{HDG-sep}
\begin{align}
	a_h(\bm{r}_h,\bm{s}_h)+
	b_h(\bm{u}_h,\widehat{\bm{u}}_h;\bm{s}_h)&=0,\label{fhm1}\\
	b_h(\bm{v}_h,\widehat{\bm{v}}_h;\bm{r}_h)+c_h(p_h,\widehat{p}_h;\bm{v}_h)
	-s^u_h(  \bm{u}_h,\widehat{\bm{u}}_h; \bm{v}_h,\widehat{\bm{v}}_h ) +k^2(\bm u_h,\bm v_h)&=-(\bm{f},\bm{v}_h),\label{fhm2}\\
	c_h(q_h,\widehat{q}_h;\bm{u}_h)+s^p_h(p_h,\widehat{p}_h; q_h,\widehat{q}_h)&=0\label{fhm3}
\end{align}
\end{subequations}
for any $(\bm{s}_h,\bm{v}_h,\widehat{\bm{v}}_h,q_h,\widehat{q}_h) \in \bm{R}_h\times\bm{U}_h\times\widehat{\bm{U}}_h\times {M}_h\times \widehat{{M}}_h$. Here the bilinear forms $a_h$, $b_h$,   $c_h$, $s_h^u$ and $s_h^p$ are defined by
\begin{align*}
a_h( {\bm{r},\bm{s}})&=(\bm{r},\bm{s}),\\ 
b_h(\bm{u},\widehat{\bm{u}};\bm{s})&= -( \bm\nabla_h\times \bm{u},\bm{s})+\langle \bm{n}\times(\bm{u}-\widehat{\bm{u}}),\bm{s} \rangle_{\partial\mathcal{T}_h},\\
c_h(q,\widehat{q};v)&= {-({k}^2+1)(\bm{v}, \nabla_h q)+({k}^2+1)\langle \bm{n}\cdot\bm{v}, q-\widehat{q} \rangle_{\partial\mathcal{T}_h}},\\
s_h^u(  \bm{u},\widehat{\bm{u}}; \bm{v},\widehat{\bm{v}} )&=\langle  {\h}^{-1} \bm{n}\times(\bm{u}-\widehat{\bm{u}}), \bm{n}\times(\bm{v}-\widehat{\bm{v}})\rangle_{\partial\mathcal{T}_h},\\
s^p_h(p,\widehat{p}; q,\widehat{q})&=({k}^2+1)\langle  {\h}^{-1}(p-\widehat{p}),q-\widehat{q} \rangle_{\partial\mathcal{T}_h}.
\end{align*}
To simplify the notation,  we introduce the spaces
\eq{\label{Sgm}
\bm{\Sigma}:=&\prod_{T\in\mathcal{T}_h}\bm H^s(T)\times\prod_{T\in\mathcal{T}_h}\bm H^s(T)\cap \bm{H}({\rm curl},T)\cap \bm{H}({\rm div},T)\\
\notag &\times \prod_{F\in\mathcal{F}_h}\bm{L}^2(F)\times\prod_{T\in\mathcal{T}_h} H^1(T)\times\prod_{F\in\mathcal{F}_h} L^2(F),\\
\label{Sgm_h}\bm{\Sigma}_h :=& \bm{R}_h\times\bm{U}_h\times\widehat{\bm{U}}_h\times M_h\times \widehat{M}_h.
}
Clearly, $\bm{\Sigma}_h\subset\bm{\Sigma}$.
Given
\eqn{\bm{\sigma} =(\bm{r}, \bm{u}, \widehat{\bm{u}}, p, \widehat{p})\in \bm{\Sigma},
\qquad \bm{\tau} &:=(\bm{s}, \bm{v}, \widehat{\bm{v}}, q, \widehat{q })\; \in \bm{\Sigma},
}  We define the following bilinear forms on $\bm{\Sigma}\times\bm{\Sigma}$ as
\begin{align}\label{eq:DefBh}
\mathcal{B}_h^{\pm}(\bm{\sigma},\bm{\tau})
:=&a_h(\bm{r},\bm{s})+b_h(\bm{u},\widehat{\bm{u}};\bm{s})\notag\\
&+b_h(\bm{v},\widehat{\bm{v}};\bm{r})+c_h(p,\widehat{p};\bm{v})+c_h(q,\widehat{q};\bm{u}) {\mp} {k}^2(\bm u,\bm v) \nonumber\\
&-s^u_h(  \bm{u},\widehat{\bm{u}}; \bm{v},\widehat{\bm{v}} )+s^p_h(p,\widehat{p}; q,\widehat{q}) \notag\\
=&(\bm{r},\bm{s})-( \bm\nabla_h\times \bm{u},\bm{s})+\langle \bm{n}\times(\bm{u}-\widehat{\bm{u}}),\bm{s} \rangle_{\partial\mathcal{T}_h}\\
&-( \bm\nabla_h\times \bm{v},\bm{r})+\langle \bm{n}\times(\bm{v}-\widehat{\bm{v}}),\bm{r} \rangle_{\partial\mathcal{T}_h}\notag\\
&-({k}^2+1)\big((\bm{v}, \nabla_h p)-\langle \bm{n}\cdot\bm{v}, p-\widehat{p} \rangle_{\partial\mathcal{T}_h}\big) \nonumber\\
&-({k}^2+1)\big((\bm{u}, \nabla_h q)-\langle \bm{n}\cdot\bm{u}, q-\widehat{q} \rangle_{\partial\mathcal{T}_h}\big) {\mp} {k}^2(\bm u,\bm v)\notag\\
&-\langle  {\h}^{-1} \bm{n}\times(\bm{u}-\widehat{\bm{u}}), \bm{n}\times(\bm{v}-\widehat{\bm{v}})\rangle_{\partial\mathcal{T}_h}+({k}^2+1)\langle  {\h}^{-1}(p-\widehat{p}),q-\widehat{q} \rangle_{\partial\mathcal{T}_h}\notag,	
\end{align}
and the linear functional on $\bm{\Sigma}$
\begin{align}
\mathcal{F}_h(\bm{\tau}):=&-(\bm{f},\bm{v}).\label{eq:DefFh}
\end{align}
Therefore, the HDG method \eqref{HDG-sep} can be rewritten in the following compact form:
\par\smallskip\noindent
Find $\bm{\sigma}_h= {(\bm{r}_h ,\bm{u}_h ,\widehat{\bm{u}}_h ,p_h ,\widehat{p}_h )}\in \bm{\Sigma}_h$ such that
\begin{eqnarray}
\mathcal{B}_h^{ -}(\bm{\sigma}_h,\bm{\tau}_h)=\mathcal{F}_h(\bm{\tau}_h)&\qquad\forall \bm{\tau}_h\in\bm{\Sigma}_h.\label{Bh_HDG}
\end{eqnarray}

\vspace{0.05in}
It can be verified that the following orthogonality property holds for the HDG scheme \eqref{Bh_HDG}.
\begin{lemma} [Orthogonality] \label{lem:Orth}
Let $(\bm r,\bm u,p)$ and $\bm{\sigma}_h\in \bm{\Sigma}_h$ be the solutions of \eqref{mixed} and \eqref{Bh_HDG}, respectively. Then we have
\begin{eqnarray}\label{or}
	\mathcal{B}_h^{-}(\bm{\sigma}-\bm{\sigma}_h,\bm{\tau}_h)=0&\qquad\forall\bm{\tau}_h\in\bm{\Sigma}_h,
\end{eqnarray}
where $\bm{\sigma}=\big(\bm r,\bm u,\bm u|_{\mathcal{F}_h},p,p|_{\mathcal{F}_h}\big)$ and $|_{\mathcal{F}_h}$ denote the trace of the function on the union of faces in $\mathcal{F}_h$.
\end{lemma}

We introduce the following mesh-dependent norms and seminorms:
\begin{subequations}
\begin{align}
	\|(\bm{v},\widehat{\bm{v}})\|^2_{\text{curl}}:=&\|\nabla_{h}\times\bm{v}\|^2_0+\| {\h}^{-1/2}\bm{n}\times(\bm{v}-\widehat{\bm{v}})\|^2_{0,\partial\mathcal{T}_h},\\
	\|(\bm{v},\widehat{\bm{v}})\|^2_{\text{div}}:=&\| {\h}\nabla_{h}\cdot\bm{v}\|^2_0+\|\h^{1/2}[\![\bm{n}\cdot\bm{v}]\!]\|^2_{0,\mathcal{F}^I_h},\\
	\|(\bm{v},\widehat{\bm{v}})\|^2_{U}:=& \|(\bm{v},\widehat{\bm{v}})\|^2_{\text{curl}}+({k}^2+1)\|(\bm{v},\widehat{\bm{v}})\|^2_{\text{div}},\\
	\|(q,\widehat{q})\|^2_{P}:=&({k}^2+1)\|\nabla_hq\|^2_0+({k}^2+1)\| \h^{-1/2}(q-\widehat{q})\|^2_{0,\partial\mathcal{T}_h}\label{def-p},\\
	\|\bm{\tau} \|^2_{\bm{\Sigma}_h}:=&\|\bm{s} \|_0^2+\|(\bm{v} ,\widehat{\bm{v} })\|^2_{U}+\|(q ,\widehat{q })\|^2_{P}+{k}^2\|\bm v\|_0^2\label{norm-lambda}\notag\\
	{=}&\|\bm{s} \|_0^2+\|\nabla_{h}\times\bm{v}\|_0^2+\|\h^{-1/2}\bm{n}\times(\bm{v}-\widehat{\bm{v}})\|^2_{0,\partial\mathcal{T}_h} \\
	&+(k^2+1)\Big(\|\h\nabla_{h}\cdot\bm{v}\|^2_0+\|\h^{1/2}[\![\bm{n}\cdot\bm{v}]\!]\|^2_{0,\mathcal{F}^I_h}\notag\\
	&+\|\nabla_hq\|^2_0+\| \h^{-1/2}(q-\widehat{q})\|^2_{0,\partial\mathcal{T}_h}\Big)+{k}^2\|\bm v\|_0^2,\notag
\end{align}
\end{subequations}
where $\bm{\tau} =(\bm{s}, \bm{v}, \widehat{\bm{v}}, q, \widehat{q })$.
%

\section{Elliptic projection}
\label{section4}
In this section, we derive an error estimate of the following elliptic  projection based on the bilinear form $\mathcal{B}_h^+$. {This result will be used to analyze the HDG method \eqref{Bh_HDG}. Given
$\bm{\sigma}:=\big(\bm r,\bm u,\bm u|_{\mathcal{F}_h}, p,p|_{\mathcal{F}_h}\big)$,} find $ \bm{\MP}_h\bm{\sigma}\in \bm{\Sigma}_h$ such that
\eq{\label{EP}
\mathcal{B}_h^+( \bm{\MP}_h\bm{\sigma},\bm{\tau}_h)=\mathcal{B}_h^+(\bm{\sigma},\bm{\tau}_h)\quad\forall\bm{\tau}_h\in\bm{\Sigma}_h.
}

\subsection{ {Approximation errors}}
We first {consider the approximation properties of the discrete space $\Sigma_h$. For any integer $j\ge 0$, denote by $\mathbb{P}_j(\mathcal T_h):=\prod_{T\in\mathcal T_h}\mathbb{P}_j(T)$ and by $\mathbb{P}_j(\mathcal F_h):=\prod_{F\in\mathcal F_h}\mathbb{P}_j(F)$. Let $\Pi^{o}_j: L^2(\Omega)\rightarrow \mathbb{P}_{j}(\mathcal T_h)$ and $\Pi^{\partial}_j: L^2(\mathcal F_h)\rightarrow \mathbb{P}_{j}(\mathcal F_h)$ be the usual $L^2$ projection operators. The following stability {  results} and error estimates hold.
\begin{lemma}[{\cite[Lemma 2.7]{MR4092295}}]\label{lemma2.5}
	For any $T\in\mathcal{T}_h$ and $F\in\mathcal{F}_h$ and $j\geq 0$, it holds
	\begin{align*}
		&\|v-\Pi^{o}_jv\|_{0,T}\lesssim  h_T^{s}|v|_{s,T},&\forall \, v\in H^{s}(T),\\
		&\|v-\Pi^{o}_jv\|_{0,\partial T}\lesssim  h_T^{s-1/2}|v|_{s,T},&\forall \,v\in H^{s}(T),\\
		&\|v-\Pi^{\partial}_jv\|_{0,\partial T}\lesssim  h_T^{s-1/2}|v|_{s,T},&\forall \,v\in H^{s}(T),\\
		&\|\Pi^{o}_jv\|_{0,T}\le\|v\|_{0,T},&\forall \, v\in L^{2}(T),\\
		&\|\Pi^{\partial}_jv\|_{0,F}\le\|v\|_{0,F},&\forall  \,v\in L^{2}(F),
	\end{align*}
	where $s\in (1/2, j+1]$.
\end{lemma}

Next, we recall the error estimate results for the interpolation operator  $\bm{\mathcal{P}}^{\rm curl}_{{\ell}}: \bm{H}^s(\text{curl};\mathcal T_h)$ $\to [\mathbb{P}_{{\ell}}(\mathcal T_h)]^3$ for the N\'ed\'elec element of second type.

\begin{lemma} [cf. \cite{MR864305,MR1609607,MR2059447,MR2194528}]\label{lemma4.2}
	There hold for $t\in (1/2, {\ell}]$ and $t^*\in (1, {\ell+1}]$,
	\begin{subequations}
		\begin{align}
			\label{curl-est0}
			{\|\bm{v}-\bm{\mathcal{P}}^{\rm curl}_{{\ell}}\bm{v}\|_{0,T}}&\lesssim h_T^{t^*}\|\bm{v}\|_{t^*},\\	
			\label{curl-est}
			\|\bm{v}-\bm{\mathcal{P}}^{\rm curl}_{{\ell}}\bm{v}\|_{0,T}&\lesssim  h_T^t\left(  {\|\bm{v}\|_{t,T}+h_T\|\bm\nabla\times\bm{v}\|_{t,T}}\right),\\
			\label{curl-est2}
			\|\bm\nabla\times(\bm{v}-\bm{\mathcal{P}}^{\rm curl}_{{\ell}}\bm{v})\|_{0,T}&\lesssim  h_T^t {\|\bm\nabla\times\bm{v}\|_{t,T}}.
		\end{align}
	\end{subequations}	%
\end{lemma}

We  next recall two lemmas which present two interpolation operators of Osward type \cite{oswald93}.  The first one indicates that every discontinuous piecewise polynomial in  $M_h$ has a good $H^1$-conforming approximation (see, e.g., \cite{oswald93,be07,Chen-2018-01,MR2300291,zhu2013preasymptotic}). {The second one says that every discontinuous piecewise polynomial in  $[\mathbb P_{\ell}(\mathcal{T}_h)]^3$ has a good $\bm H({\rm curl})$-conforming approximation.}
\begin{lemma}\label{Plc}  There exists an interpolation operator  {${\Pi_{\ell+1}^{\rm c}}: M_h\to M_h\cap H^1_0(\Omega)$} such that
	\begin{align*}
		\|\nabla_h q_h-\nabla {\Pi_{\ell+1}^{\rm c}} q_h\|_0\lesssim  \|\h^{-1/2}[\![q_h]\!]\|_{\mathcal{F}_h}\quad \forall q_h\in M_h.
	\end{align*}
\end{lemma}
Note that the supscript $^{\rm c}$  stands for ``conforming''. 
\begin{lemma}[cf. {\cite[Proposition 4.5]{MR2194528}}]\label{pic}
	There is an interpolation $\bm{\Pi}_h^{\rm curl,c}$ from $[\mathbb P_{\ell}(\mathcal{T}_h)]^3$
	to $
	[\mathbb P_{\ell}(\mathcal{T}_h)]^3\cap \bm H_0({\rm curl};\Omega)
	$ such that for all $\bm v_h\in [\mathbb P_{\ell}(\mathcal{T}_h)]^3$, we have the following approximation properties
	\begin{align*}
		\|\bm{\Pi}_h^{\rm curl,c}\bm v_h-\bm v_h\|_0&\lesssim
		\| {\h}^{1/2}\bm n\times[\![\bm v_h]\!]\|_{0,\mathcal{F}_h},\\
		\|\bm\nabla\times\bm{\Pi}_h^{\rm curl,c}\bm v_h- \bm\nabla_h\times \bm v_h\|_0&\lesssim
		\| {\h}^{-1/2}\bm n\times[\![
		\bm v_h
		]\!]\|_{0,\mathcal{F}_h}.
	\end{align*}
\end{lemma}

The following lemma shows that every discrete function in $\bm H_0({\rm curl};\Omega)\cap \bm U_h$ has a discrete Helmholtz decomposition  and the discrete divergence free part in the decomposition has a good ``continuous'' approximation.  (see, e.g., \cite[Theorem 4.1 and Lemma 4.5]{MR2009375}, \cite[\S7.2.1]{MR2059447})
\begin{lemma} \label{Hip}
	For any $\bm w_h\in \bm H_0({\rm curl};\Omega)\cap \bm U_h$, there exist $\bm{z}_h\in \bm H_0({\rm curl};\Omega)\cap \bm U_h$ and $\xi_h\in H_0^1(\Omega)\cap \mathbb P_{\ell+1}(\mathcal T_h)$  such that
	\eq{\label{DHD}
		\bm{w}_h=\bm{z}_h+\nabla \xi_h, \quad (\bm{z}_h,\nabla q_h)=0\quad\forall q_h\in H_0^1(\Omega)\cap \mathbb P_{\ell+1}(\mathcal T_h).
	}
	Moreover, there exists  $\bm\Theta \in \bm H_0({\rm curl};\Omega)\cap \bm H({\rm div^0 };\Omega)$  such that
	\begin{align}\label{TA}
		\bm\nabla\times \bm\Theta =\bm\nabla\times \bm w_h,\; \|\bm\Theta \|_{s}\lesssim  \|\bm\nabla\times \bm w_h\|_0,\;\text{and }
		\|\bm z_h-\bm\Theta \| _0\lesssim
		h^{s}\|\bm\nabla\times \bm w_h\|_0,
	\end{align}
	where $s\in (1/2,1]$ is the regularity constant from \Cref{reg} which is dependent on $\Omega$.
\end{lemma}

Now we consider the approximation properties of the discrete space $\bm{\Sigma}_h$.  Given sufficiently smooth $\bm r,\bm u,$ and $p$ satisfying the boundary conditions $\bm{u}\times\bm n=0$ and $p=0$ on $\Gamma$, let $\bm{\sigma}:=\big(\bm r,\bm u,\bm u|_{\mathcal{F}_h}, p,p|_{\mathcal{F}_h}\big)$ and define its approximation in $\bm{\Sigma}_h$ by
%
\begin{align*}
	\bm{\mathcal I}_h\bm{\sigma}:=\big(\bm{\Pi}^o_{m}\bm{r},\bm{\mathcal{P}}^{\rm curl}_{\ell}\bm{u},  {\bm n\times\bm{\mathcal{P}}^{\rm curl}_{\ell}\bm{u}\times\bm n|_{\mathcal{F}_h}},\Pi^{o}_{\ell+1}p,\Pi^{\partial}_{\ell+1} {(p|_{\mathcal{F}_h})}\big).
\end{align*}
The following lemma gives the error estimate of $ \bm{\mathcal I}_h$ in the norm $\|\cdot\|_{\bm{\Sigma}_h}$.
\begin{lemma} \label{l53}
	Assume that $(\bm{r} =\bm\nabla\times\bm u,\bm{u},p)\in [H^t(\Omega)]^3\times  \big([H^t(\Omega)]^3\cap\bm{H}_0({\rm curl};\Omega)\big) \times \big(H^{t+1}(\Omega)\cap H_0^1(\Omega)\big)$ with $t\in (1/2,\ell]$ and that $\nabla\cdot\bm{u}=0$. Then there holds
	\begin{align*}
		\|\bm{\sigma}-\bm{\mathcal I}_h\bm{\sigma}\|_{\bm{\Sigma}_h}\lesssim  h^{t}
		{\big(
			(1+kh)
			{\|\bm{r}\|_t}+({k}+1){\|\bm{u}\|_t}+({k}+1){\|p\|_{t+1}}\big).}
	\end{align*}
\end{lemma}

\begin{proof}  From the definition \eqref{norm-lambda} of the norm $\|\cdot\|_{\bm{\Sigma}_h}$ we have
	\eq{\label{I0}
		\|\bm{\sigma}-\bm{\mathcal I}_h\bm{\sigma}\|_{\bm{\Sigma}_h}^2=
		&\Big(\|\bm{r}-\bm{\Pi}^o_{m}\bm{r}\|_0^2+\|\nabla_{h}\times(\bm{u}-\bm{\mathcal{P}}^{\rm curl}_{\ell}\bm{u})\|_0^2+{k}^2\|(\bm{u}-\bm{\mathcal{P}}^{\rm curl}_{\ell}\bm{u})\|_0^2\Big)\notag\\
		&+(k^2+1)\Big(\|\h\nabla_{h}\cdot(\bm{u}-\bm{\mathcal{P}}^{\rm curl}_{\ell}\bm{u})\|^2_0+\|\h^{1/2}[\![\bm{n}\cdot(\bm{u}-\bm{\mathcal{P}}^{\rm curl}_{\ell}\bm{u})]\!]\|^2_{0,\mathcal{F}^I_h}\Big)\notag\\
		&+(k^2+1)\Big(\|\nabla_h(p-\Pi^{o}_{\ell+1}p)\|^2_0+\| \h^{-1/2}(\Pi^{o}_{\ell+1}p-\Pi^{\partial}_{\ell+1}p)\|^2_{0,\partial\mathcal{T}_h}\Big)\notag\\
		=:&I^2+II^2+III^2.
	}
	We will make the estimation for the three terms $I, II, III$, respectively. Firstly, from \Cref{lemma2.5} and \Cref{lemma4.2}, we have
	\eq{\label{I1}
		I\lesssim h^{t}{\|\bm{r}\|_t}+(1+kh)h^t{\|\bm\nabla\times\bm{u}\|_t}+kh^t{\|\bm{u}\|_t}.
	}
	Secondly, noting that $\nabla\cdot\bm{u}=0$, from the trace inequality, the inverse inequality,  \Cref{lemma2.5}, and \Cref{lemma4.2}, we conclude that
	\begin{align}\label{I2}
		\begin{split}
			II \lesssim &(k+1)\Big(\|\h\nabla\cdot(\bm{u}-\bm{\Pi}^{\rm div}_{ \ell}\bm{u})\|_0+\|\h^{1/2}[\![\bm{n}\cdot(\bm{u}-\bm{\Pi}^o_{ \ell}\bm{u})]\!]\|_{0,\mathcal{F}^I_h}\\
			&+\|\h\nabla_{h}\cdot(\bm{\Pi}^{\rm div}_{ \ell}\bm{u}-\bm{\mathcal{P}}^{\rm curl}_{\ell}\bm{u})\|_0+\|\h^{1/2}[\![\bm{n}\cdot(\bm{\Pi}^o_{ \ell}\bm{u}-\bm{\mathcal{P}}^{\rm curl}_{\ell}\bm{u})]\!]\|_{0,\mathcal{F}^I_h}\Big)\\
			\lesssim&(k+1)\big(h^t{\|\bm{u}\|_t}
			+\|\bm{\Pi}^{\rm div}_{ \ell}\bm{u}-\bm{\mathcal{P}}^{\rm curl}_{\ell}\bm{u}\|_0+\|\bm{\Pi}^o_{ \ell}\bm{u}-\bm{\mathcal{P}}^{\rm curl}_{\ell}\bm{u}\|_0
			\big)\\
			\lesssim&(k+1)\big(h^t{\|\bm{u}\|_t}+h^{t+1}{\|\bm\nabla\times\bm{u}\|_t}\big),
		\end{split}
	\end{align}
	where $\bm{\Pi}_{\ell}^{\rm div}$ is the $\bm H({\rm div};\Omega)$-conforming interpolation operator onto $\bm H({\rm div};\Omega)\cap \mathbb P_{\ell}(\mathcal T_h)$, and we have used $\nabla\cdot\bm{\Pi}_{\ell}^{\rm div}\bm u=0$ (see e.g. \cite{MR2059447}) to derive the penultimate inequality.
	Thirdly, it follows from   \Cref{lemma2.5} that
	\eq{\label{I3}
		III\lesssim (k+1)h^t{\|p\|_{t+1}}. 
	}
	Then the proof of the theorem follows by plugging \eqref{I1}--\eqref{I3} into \eqref{I0}.
\end{proof}

The following lemma shows that   $\bm{\sigma}-\bm{\mathcal I}_h\bm{\sigma}$ satisfies an approximate Galerkin orthogonality with respect to the bilinear form $\mathcal{B}_h^{+}$.
\begin{lemma}\label{lemma47}
	Assume that  $(\bm{r} =\bm\nabla\times\bm u,\bm{u},p)\in [H^t(\Omega)]^3\times  \big([H^t(\Omega)]^3\cap\bm{H}_0({\rm curl};\Omega)\big) \times \big(H^{t+1}(\Omega)\cap H_0^1(\Omega)\big)$ with $t\in (1/2,\ell]$ and that $\nabla\cdot\bm{u}=0$. Let $\bm{\sigma}=\big(\bm r,\bm u,\bm u|_{\mathcal{F}_h}, p,p|_{\mathcal{F}_h}\big)$. Then
	\begin{align}
		|\mathcal{B}_h^{+}(\bm{\sigma}-\bm{\mathcal I}_h\bm{\sigma},\bm{\tau})|\lesssim   h^{t}\big(
		(1+kh){\|\bm{r}\|_t}+({k}+1){\|\bm{u}\|_t}
		+({k}+1){\|p\|_{t+1}}\big) \|\bm{\tau}\|_{\bm{\Sigma}_h} \qquad \forall \bm{\tau}\in\bm{\Sigma}.\notag
	\end{align}
\end{lemma}

\begin{proof}
	For any  $\bm{\tau}=(\bm s,\bm v,\widehat{\bm v},q,\widehat q)\in\bm{\Sigma}$, it follows from the definition of $\mathcal{B}_h^+$ in \eqref{eq:DefBh},  integration by parts, and the identity $\langle \bm{n}\cdot\bm{v}, p-\Pi^{\partial}_{\ell+1}p \rangle_{\partial\mathcal{T}_h}=\langle [\![\bm{n}\cdot\bm{v}]\!], p-\Pi^{\partial}_{\ell+1}p \rangle_{\mathcal{F}_h^I}$  that
	\begin{align*}
		\mathcal{B}_h^{+}(\bm{\sigma}-\bm{\mathcal I}_h\bm{\sigma},\bm{\tau})
		=&(\bm{r}-\bm{\Pi}^o_m\bm{r},\bm{s})-( \bm\nabla_h\times (\bm{u}-\bm{\mathcal{P}}_{\ell}^{\rm curl}\bm{u}),\bm{s})- {k}^2(\bm u-\bm{\mathcal{P}}_{\ell}^{\rm curl}\bm{u},\bm v)\\
		&-( \bm\nabla_h\times \bm{v},\bm{r}-\bm{\Pi}^o_m\bm{r})+\langle \bm{n}\times(\bm{v}-\widehat{\bm{v}}),\bm{r}-\bm{\Pi}^o_m\bm{r} \rangle_{\partial\mathcal{T}_h}\notag\\
		&+({k}^2+1)\big((\nabla_h\cdot\bm{v}, p-\Pi^{o}_{\ell+1}p )-\langle [\![\bm{n}\cdot\bm{v}]\!], p-\Pi^{\partial}_{\ell+1}p \rangle_{\mathcal{F}_h^I}\big) \nonumber\\
		&-({k}^2+1)\big((\bm{u}-\bm{\mathcal{P}}_{\ell}^{\rm curl}\bm{u}, \nabla_h q)-\langle \bm{n}\cdot(\bm{u}-\bm{\mathcal{P}}_{\ell}^{\rm curl}\bm{u}), q-\widehat{q}\rangle_{\partial\mathcal{T}_h}\big) \notag\\
		&+({k}^2+1)\langle \h^{-1}(-\Pi^{o}_{\ell+1}p+\Pi^{\partial}_{\ell+1}p),q-\widehat{q} \rangle_{\partial\mathcal{T}_h}\notag.	
	\end{align*}
	Therefore, from the Cauchy-Schwarz inequality,  \eqref{I0}, and \eqref{norm-lambda} , we conclude that
	\eqn{
		&|\mathcal{B}_h^{+}(\bm{\sigma}-\bm{\mathcal I}_h\bm{\sigma},\bm{\tau})|\lesssim
		\big(\|\bm{\sigma}-\bm{\mathcal I}_h\bm{\sigma}\|_{\bm{\Sigma}_h}+\|\h^\frac12( \bm{r}-\bm{\Pi}^o_m\bm{r})\|_{0,\partial\mathcal{T}_h}+\|(\bm{u}-\bm{\mathcal{P}}^{\rm curl}_{\ell}\bm{u})\|_0\\
		&+(k^2+1)\big(\|\h^{1/2}\bm{n}\cdot(\bm{u}-\bm{\mathcal{P}}^{\rm curl}_{\ell}\bm{u})\|^2_{0,\Gamma}+\|\h^{-1}(p-\Pi^{o}_{\ell+1}p)\|_0+\|\h^{-\frac12}(p-\Pi^{o}_{\ell+1}p)\|_{0,\partial\mathcal{T}_h}\big)\big)\|\bm{\tau}\|_{\bm{\Sigma}_h}\cb{,}
	}
	which together with \Cref{l53,lemma2.5,lemma4.2} completes the proof of the lemma.
\end{proof}
%

\subsection{Discrete inf-sup condition}  In this subsection we show that $\mathcal{B}_h^+$ satisfies a discrete inf-sup condition.
\begin{theorem} [Discrete inf-sup condition]\label{Th45}
	For all $\bm\sigma_h\in\bm{\Sigma}_h$, the bilinear form $\mathcal{B}^+_h(\cdot,\cdot)$ defined in \eqref{eq:DefBh} satisfies
	\begin{align}\label{inf-sup1}
		\sup_{\bm 0\neq \bm{\tau}_h\in\bm{\Sigma}_h}\frac{\mathcal{B}^+_h(\bm{\sigma}_h,\bm{\tau}_h)}{\|\bm{\tau}_h\|_{\bm{\Sigma}_h}}
		=\sup_{\bm 0\neq \bm{\tau}_h\in\bm{\Sigma}_h}\frac{\mathcal{B}^+_h(\bm{\tau}_h,\bm{\sigma}_h)}{\|\bm{\tau}_h\|_{\bm{\Sigma}_h}}\ge \beta  \|\bm{\sigma}_h\|_{\bm{\Sigma}_h},
	\end{align}
	where $\beta$ is a constant independent of $k$ and $h$.
\end{theorem}

\begin{proof}
	The proof consists of five steps.
	\vspace{0.1in}
	
	\textbf{Step one:}

	Let $\bm{\tau}^1_h=(\bm{r}_h ,-\bm{u}_h ,-\widehat{\bm{u}}_h ,p_h ,\widehat{p}_h )\in\bm{\Sigma}_h$. By \eqref{norm-lambda} and \eqref{eq:DefBh}, we have
	\begin{equation}\label{tau1}
		\|\bm{\tau}_h^1\|_{\bm{\Sigma}_h}=\|\bm{\sigma}_h\|_{\bm{\Sigma}_h},
	\end{equation}
	and
	\begin{align}\label{B1}
		\mathcal{B}_h^+(\bm{\sigma}_h,\bm{\tau}^1_h)&=\|\bm{r}_h\|^2_0+\|\h^{-1/2}\bm{n}\times(\bm{u}_h-\widehat{\bm{u}}_h)\|^2_{0,\partial\mathcal{T}_h}\nonumber\\
		&\quad+ (k^2+1)\|\h^{-1/2}(p_h-\widehat{p}_h)\|^2_{0,\partial\mathcal{T}_h}+{k}^2\|\bm u_h\|_0^2.
	\end{align}
	\vspace{0.05in}

	\textbf{Step two:}
	
	Let $\bm{\tau}^2_h=(- \bm\nabla_h\times \bm{u}_h,\bm{0},\bm{0},0,0) \in\bm{\Sigma}_h$. By \eqref{norm-lambda} and \eqref{eq:DefBh},  we have
	\begin{align}\label{tau2}
		\|\bm{\tau}_h^2\|_{\bm{\Sigma}_h} =\| \bm\nabla_h\times \bm{u}_h\|_0\le\|\bm{\sigma}_h\|_{\bm{\Sigma}_h},
	\end{align}
	and
	\begin{align}\label{B2}
		\begin{split}
			\mathcal{B}_h^+(\bm{\sigma}_h,\bm{\tau}^2_h)&=-(\bm{r}_h, \bm\nabla_h\times \bm{u}_h)+\| \bm\nabla_h\times \bm{u}_h\|^2_0-\langle \bm\nabla_h\times \bm{u}_h,\bm{n}\times(\bm{u}_h- \widehat{\bm{u}}_h  )  \rangle_{\partial\mathcal{T}_h}\\
			&\ge \frac{1}{2}\| \bm\nabla_h\times \bm{u}_h\|^2_0-C_1\|\bm{r}_h\|^2_0-C_2 \|\h^{-1/2}\bm{n}\times(\bm{u}_h-\widehat{\bm{u}}_h)\|^2_{0,\partial\mathcal{T}_h},
		\end{split}
	\end{align}
	where we have used the inverse trace inequality $\|\h^\frac12 \bm\nabla_h\times \bm{u}_h\|_{0,\partial\mathcal{T}_h}\lesssim \| \bm\nabla_h\times \bm{u}_h\|_0$ and the Young's inequality to derive the last inequality.
	\vspace{0.05in}

	\textbf{Step three:}
	
	Let $\bm{\tau}^3_h=(\bm{0},\bm{0},\bm{0},w_h,\widehat{w}_h)\in\bm{\Sigma}_h$, where $w_h= {\h}^{2}\nabla_h\cdot\bm{u}_h$ on $\mathcal{T}_h$ and let $\widehat{w}_h=- {\h}[\![\bm{n}\cdot \bm{u}_h]\!]$ on $\mathcal{F}_h^I$ and $\widehat{w}_h=0$ on $\Gamma$.   By \eqref{norm-lambda} and the inverse inequality, we have
	\begin{align}\label{tau3}
		\begin{split}
			\|\bm{\tau}^3_h\|_{\bm{\Sigma}_h}^2
			&= (k^2+1)\big(\|\nabla_h w_h\|^2_{0}+\|\h^{-1/2}(w_h-\widehat{w}_h)\|^2_{\partial\mathcal{T}_h}\big)\\
			&\lesssim   (k^2+1)\| \h\nabla_h\cdot\bm{u}_h\|^2_0+ (k^2+1)\|\h^{1/2}[\![\bm{n}\cdot\bm{u}_h]\!]\|^2_{0,\mathcal{F}^I_h}\\
			&\lesssim\|\bm{\sigma}_h\|_{\bm{\Sigma}_h}^2.
		\end{split}
	\end{align}
	By  \eqref{eq:DefBh} and $-(\bm{u_h}, \nabla_h w_h)+\langle \bm{n}\cdot\bm{u}_h, w_h-\widehat{w}_h \rangle_{\partial\mathcal{T}_h}=(\nabla_h \cdot \bm{u}_h, w_h)-\langle \bm{n}\cdot\bm{u}_h, \widehat{w}_h \rangle_{\partial\mathcal{T}_h}$,  we have
	\begin{align}\label{B3}
		\begin{split}
			\mathcal{B}^+_h(\bm{\sigma}_h,\bm{\tau}^3_h)& =  (k^2+1)\Big(\|\h\nabla\cdot\bm{u}_h\|^2_0+ \| \h^{1/2}[\![\bm{n}\cdot\bm{u}_h]\!]\|^2_{0,\mathcal{F}_h^I}\Big)\\
			&\quad+(k^2+1) \langle \h^{-1} (p_h-\widehat{p}_h), w_h-\widehat{w}_h\rangle_{\partial\mathcal{T}_h} \\
			&\ge\frac{k^2+1}{2}\Big(\|\h\nabla\cdot\bm{u}_h\|^2_0+ \| \h^{1/2}[\![\bm{n}\cdot\bm{u}_h]\!]\|^2_{0,\mathcal{F}_h^I}\Big) \\
			&\quad-C_3 (k^2+1)\|\h^{-1/2}(p_h-\widehat{p}_h )\|^2_{0,\mathcal{F}_h}.
		\end{split}
	\end{align}
	\vspace{0.05in}
	
	\textbf{Step four:}
	
	Let $\bm{\tau}^4_h=(\bm{0},-\nabla{\Pi_{\ell+1}^{\rm c}} p_h,-\bm{n}\times\nabla{\Pi_{\ell+1}^{\rm c}} p_h\times\bm{n},0,0)\in\bm{\Sigma}_h$. By \eqref{norm-lambda}, the inverse inequality,  {and \Cref{Plc}}, we have
	\begin{align}\label{tau4}
		\begin{split}
			\|\bm{\tau}^4_h\|_{\bm{\Sigma}_h}^2
			&=  {(k^2+1)}\big(\|\h\nabla_h\cdot \nabla{\Pi_{\ell+1}^{\rm c}} p_h\|^2_0+ \| {\h^{1/2}}[\![\nabla{\Pi_{\ell+1}^{\rm c}} p_h\cdot\bm{n}]\!]\|_{0,\mathcal{F}^I_h}\big)
			+{k}^2\| \nabla{\Pi_{\ell+1}^{\rm c}} p_h\|_0^2
			\\
			&\lesssim  ({k}^2+1)\|\nabla{\Pi_{\ell+1}^{\rm c}} p_h\|^2_0
			{\lesssim  ({k}^2+1)\big(\|\nabla_h  p_h\|^2_0+ \|\h^{-1/2}[\![p_h]\!]\|_{\mathcal{F}_h}\big)}\\
			&=    ({k}^2+1)\big(\|\nabla_h  p_h\|^2_0+ \|\h^{-1/2}[\![p_h-\widehat{p}_h]\!]\|_{\mathcal{F}_h}\big) \\
			&\lesssim  \|\bm{\sigma}_h\|_{\bm{\Sigma}_h}^2.
		\end{split}
	\end{align}
	It follows from \eqref{eq:DefBh}, \Cref{Plc},   the Cauthy-Schwarz's inequality, and the Young's inequality that
	\begin{align}\label{B4}
		\mathcal{B}^+_h(\bm{\sigma}_h,\bm{\tau}^4_h)=& {({k}^2+1)\big((\nabla{\Pi_{\ell+1}^{\rm c}} p_h, \nabla_h p_h)-\langle \bm{n}\cdot\nabla{\Pi_{\ell+1}^{\rm c}} p_h, p_h-\widehat{p} _h\rangle_{\partial\mathcal{T}_h}\big)}
		+{k}^2(\bm u_h,\nabla{\Pi_{\ell+1}^{\rm c}} p_h) \nonumber\\
		=&({k}^2+1)\big(( \nabla_h p_h, \nabla_h p_h)+( \nabla{\Pi_{\ell+1}^{\rm c}} p_h- \nabla_hp_h,\nabla_h p_h)\big)\nonumber\\
		&\quad+  ({k}^2+1)\langle   \bm{n}\cdot\nabla{\Pi_{\ell+1}^{\rm c}} p_h,\widehat{p}_h -p_h \rangle_{\partial\mathcal{T}_h}
		+{k}^2(\bm u_h,\nabla{\Pi_{\ell+1}^{\rm c}} p_h)
		\\
		\ge& \frac{{k}^2+1}{2}\|\nabla_h p_h\|^2_0-C_4({k}^2+1)  \| \h^{-1/2}(p_h-\widehat{p}_h)\|^2_{0,\partial\mathcal{T}_h}
		-C_5{k}^2\|\bm u_h\|^2_0.\nonumber
	\end{align}
	\vspace{0.05in}

	\textbf{Step five:}
	
	Let $\bm{\tau}_h=\alpha\bm{\tau}_1+\bm{\tau}_2+\bm{\tau}_3+\bm{\tau}_4$ with $\alpha=\max(C_1,C_2,C_5,C_3+C_4)+\frac 1 2$. By \eqref{tau1}, \eqref{tau2}, \eqref{tau3} and \eqref{tau4}, we can get
	\begin{align*}
		\|\bm{\tau}_h\|_{\bm{\Sigma}_h}\lesssim  \|\bm{\sigma}_h\|_{\bm{\Sigma}_h}.
	\end{align*}
	Moreover, by combining \eqref{B1}, \eqref{B2}, \eqref{B3} and \eqref{B4},  {and \eqref{norm-lambda}}, we have
	\begin{align*}
		\mathcal{B}^+_h(\bm{\sigma}_h,\bm{\tau}_h) \ge\frac{1}{2}\|\bm{\sigma}_h\|^2_{\bm{\Sigma}_h}.
	\end{align*}
	The last two inequalities together lead to
	\begin{align*}
		\mathcal{B}^+_h(\bm{\sigma}_h,\bm{\tau}_h)\gtrsim \|\bm{\sigma}_h\|_{\bm{\Sigma}_h}\|\bm{\tau}_h\|_{\bm{\Sigma}_h},
	\end{align*}
	which implies \eqref{inf-sup1}. This completes the proof of \Cref{Th45}.
\end{proof}
%

\subsection{Error estimates of the elliptic projection}
\begin{theorem}\label{EPerr} Assume that  $(\bm{r} =\bm\nabla\times\bm u,\bm{u},p)\in [H^t(\Omega)]^3\times \big( [H^t(\Omega)]^3\cap\bm{H}_0({\rm curl};\Omega)\big) \times \big(H^{t+1}(\Omega)\cap H_0^1(\Omega)\big)$ with $t\in (1/2,\ell]$ and that $\nabla\cdot\bm{u}=0$. Let $\bm{\sigma}:=\big(\bm r,\bm u,\bm u|_{\mathcal{F}_h}, p,p|_{\mathcal{F}_h}\big)$ and let $ \bm{\MP}_h\bm{\sigma}=\big(\bm r_h^P,\bm u_h^P,\widehat{\bm u}_h^P, p_h^P,\widehat{p}_h^P\big)$ be its elliptic projection defined in \eqref{EP}. Then
	\begin{subequations}
		\begin{align}
			\|\bm{\sigma}- \bm{\MP}_h\bm{\sigma}\|_{\bm{\Sigma}_h}\lesssim&  h^{t}\big( (1+kh){\|\bm{r}\|_t}+({k}+1){\|\bm{u}\|_t}+({k}+1){\|p\|_{t+1}}\big).\label{EP1}
		\end{align}
		Moreover, for the regularity constant $s\in (1/2,1]$ from \Cref{reg}, there holds
		\begin{align}
			\begin{split}
				\|\bm u-\bm u_h^P\|_0\lesssim& (1+kh) h^{t+s}\big( (1+kh){\|\bm{r}\|_t}+({k}+1){\|\bm{u}\|_t}+({k}+1){\|p\|_{t+1}}\big) \label{EP2}
				\\
				& + \|\bm u-\bm{\mathcal{P}}_{\ell}^{\rm curl}\bm u\|_0 .
			\end{split}
		\end{align}
	\end{subequations}
\end{theorem}
\begin{proof}
	From the inf-sup condition of \Cref{Th45} and the definition of the elliptic projection \eqref{EP}, we have
	\eqn{\|\bm{\mathcal I}_h\bm{\sigma}- \bm{\MP}_h\bm{\sigma}\|_{\bm{\Sigma}_h}\lesssim &
		\sup_{\bm 0\neq \bm{\tau}_h\in\bm{\Sigma}_h}\frac{\mathcal{B}^+_h(\bm{\mathcal I}_h\bm{\sigma}- \bm{\MP}_h\bm{\sigma},\bm{\tau}_h)}{\|\bm{\tau}_h\|_{\bm{\Sigma}_h}}
		=\sup_{\bm 0\neq \bm{\tau}_h\in\bm{\Sigma}_h}\frac{\mathcal{B}^+_h(\bm{\mathcal I}_h\bm{\sigma}-\bm{\sigma},\bm{\tau}_h)}{\|\bm{\tau}_h\|_{\bm{\Sigma}_h}}
	}
	which implies \eqref{EP1}  by using   \Cref{lemma47,l53},  and the triangle inequality.  It remains to prove \eqref{EP2}. Denote by $\bm\eta:=\bm u-\bm u_h^P$. We have the following decomposition:
	\begin{align}\label{et1}
		\begin{split}
			\|\bm\eta\|_0^2=&(\bm u-\bm{\mathcal{P}}_{\ell}^{\rm curl}\bm u,\bm\eta)+(\bm{\Pi}_{\ell}^{\rm curl,c}\bm u_h^P-\bm u_h^P,\bm\eta)+(\bm w_h-\bm\Theta,\bm\eta)+(\bm\Theta,\bm\eta),
		\end{split}
	\end{align}
	where
	$\bm w_h:=\bm{\mathcal{P}}_{\ell}^{\rm curl}\bm u-\bm{\Pi}_{\ell}^{\rm curl,c}\bm u_h^P$ and
	$\bm\Theta$ is defined in \Cref{Hip}. Let $\bm w_h$ be decomposed as \eqref{DHD} in \Cref{Hip}. For any $q_h\in H_0^1(\Omega)\cap M_h$, by taking $\bm\tau_h=(\bm{0},\nabla q_h,\bm{n}\times\nabla q_h\times\bm{n},0,0)$ in \eqref{EP} and using the definition \eqref{eq:DefBh} of $\mathcal{B}^+_h$, we conclude that $(\bm u_h^P,\nabla q_h)=0$, that is, $\bm u_h^P$ is discrete divergence free.  Noting that $\nabla\cdot\bm u=0$, we have
	\eqn{(\bm w_h-\bm\Theta,\bm\eta)=(\bm{z}_h+\nabla \xi_h-\bm\Theta,\bm\eta)=(\bm{z}_h-\bm\Theta,\bm\eta).}
	It follows from \Cref{pic,Hip}, the triangle inequality, and \eqref{norm-lambda} that
	\begin{align}\label{et2}
		\begin{split}
			&(\bm{\Pi}_{\ell}^{\rm curl,c}\bm u_h^P-\bm u_h^P,\bm\eta)+(\bm w_h-\bm\Theta,\bm\eta)\\
			&\qquad\lesssim \big(\|\h^{1/2}\bm n\times[\![\bm u_h^P]\!]\|_{0,\mathcal{F}_h}+h^{s}\|\bm\nabla\times \bm w_h\|_0\big)\|\bm\eta\|_0\\
			&\qquad\lesssim\big(\|\h^{s-1/2}\bm n\times[\![\bm u_h^P]\!]\|_{0,\mathcal{F}_h}+h^{s}\| \bm\nabla_h\times  (\bm{\mathcal{P}}_{\ell}^{\rm curl}\bm u-\bm u_h^P)\|_0\big)\|\bm\eta\|_0\\
			&\qquad\lesssim h^{s}\big(\|\bm{\sigma}- \bm{\MP}_h\bm{\sigma}\|_{\bm{\Sigma}_h}+\|\bm\nabla\times (\bm{\mathcal{P}}_{\ell}^{\rm curl}\bm u-\bm u)\|_0\big)\|\bm\eta\|_0.
		\end{split}
	\end{align}
	To estimate the last term in \eqref{et1}, we introduce the following dual problem: Find $(\bm r^d,\bm u^d ,p^d)$ such that
	\begin{subequations}\label{eq:dualE}
		\begin{align}
			\bm{r}^d-\bm\nabla\times\bm{u}^d&=0&\text{ in }\Omega, \label{dualE}\\
			\bm\nabla\times \bm{r}^d+{k}^2\bm u^d+({k}^2+1)\nabla p^d&=\bm\Theta&\text{ in }\Omega,  \\
			\nabla\cdot\bm{u}^d&=0& \text{ in }\Omega,  \\
			\bm{n}\times\bm{u}^d &=\bm{0}& \text{ on }\Gamma,\\
			p^d&=0& \text{ on }\Gamma.
		\end{align}
	\end{subequations}
	Note that the above dual problem is positive definite since the sign in front of  $k^2$ in the second equation is positive (cf. \eqref{mix0source}).
	Similar to \Cref{reg}, we have the following regularity estimate of problem \eqref{eq:dualE}:
	\begin{align*}
		\|\bm{r}^d\|_{s}+(k+1)\|\bm{u}^d\|_{s} \lesssim  \|\bm\Theta\|_0, \qquad  p^d=0,
	\end{align*}
	where the regularity index $s\in (1/2,1]$ depends on $\Omega$. Denote by $\bm{\sigma}^d =(\bm{r}^d, \bm{u}^d, \bm{u}^d|_{\mathcal{F}_h}, $ $p^d, p^d|_{\mathcal{F}_h})$. Following from a similar derivation as in \Cref{S:HDG},  we conclude that $\bm{\sigma}^d$ satisfies the following variational  equation
	\eqn{
		\mathcal{B}_h^+(\bm{\sigma}^d,\bm{\tau})=-(\bm\Theta,\bm{v})\qquad \forall \bm{\tau}=(\bm{s}, \bm{v}, \widehat{\bm{v}}, q, \widehat{q})\in\bm{\Sigma}.
	}
	By taking $\bm\tau=\bm{\sigma}- \bm{\MP}_h\bm{\sigma}$ and using \eqref{EP}, the symmetry of $\mathcal{B}_h^+$, and \Cref{lemma47}, we have
	\begin{align}\label{et3}
		\begin{split}
			(\bm\Theta,\bm\eta)=&\mathcal{B}_h^+(\bm{\sigma}^d, \bm{\MP}_h\bm{\sigma}-\bm{\sigma})=\mathcal{B}_h^+(\bm{\MP}_h\bm{\sigma}-\bm{\sigma},\bm{\sigma}^d)=\mathcal{B}_h^+(\bm{\sigma}^d-\bm{\mathcal I}_h\bm{\sigma}^d, \bm{\MP}_h\bm{\sigma}-\bm{\sigma})\\
			\lesssim& h^s\big((1+kh){\|\bm{r}^d\|_s}+({k}+1){\|\bm{u}^d\|_s}\big) \|\bm{\sigma}- \bm{\MP}_h\bm{\sigma}\|_{\bm{\Sigma}_h}\\
			\lesssim&(1+kh)h^s\|\bm\Theta\|_0 \|\bm{\sigma}- \bm{\MP}_h\bm{\sigma}\|_{\bm{\Sigma}_h}.
		\end{split}
	\end{align}
	On the other hand, from \Cref{Hip,pic}, $ \bm w_h=\bm{\mathcal{P}}_{\ell}^{\rm curl}\bm u-\bm u+\bm \eta+\bm u_h^P-\bm{\Pi}_{\ell}^{\rm curl,c}\bm u_h^P$,  and \eqref{norm-lambda}, we have
	\begin{align}\label{et4}
		\begin{split}
			\|\bm\Theta\|_0\le& \|\bm\Theta-\bm z_h\|_0+\|\bm z_h\|_0
			\lesssim  h^s\|\bm\nabla\times \bm w_h\|_0+\|\bm w_h\|_0\\
			\lesssim & h^s\big(\|\bm\nabla\times (\bm u-\bm{\mathcal{P}}_{\ell}^{\rm curl}\bm u)\|_0+\|\bm{\sigma}- \bm{\MP}_h\bm{\sigma}\|_{\bm{\Sigma}_h}\big)+\|\bm u-\bm{\mathcal{P}}_{\ell}^{\rm curl}\bm u\|_0+\|\bm\eta\|_0.
		\end{split}
	\end{align}
	By combining \eqref{et1},\eqref{et2},\eqref{et3}, and \eqref{et4}, we obtain
	\eqn{\|\bm\eta\|_0\lesssim\|\bm u-\bm{\mathcal{P}}_{\ell}^{\rm curl}\bm u\|_0+(1+kh)h^{s}\big(\|\bm{\sigma}- \bm{\MP}_h\bm{\sigma}\|_{\bm{\Sigma}_h}+\|\bm\nabla\times (\bm u-\bm{\mathcal{P}}_{\ell}^{\rm curl}\bm u)\|_0\big), }
	which together with \eqref{EP1} and \eqref{curl-est2} implies \eqref{EP2}.
\end{proof}

	\section{Error estimates of the HDG methods}\label{section5}  In this section, we derive error estimates for the HDG method \eqref{oror} (or \eqref{Bh_HDG}) by using a modified duality argument, which was first used to derive the preasymptotic error estimates for the continuous interior penalty finite element method \cite{zhu2013preasymptotic}.
	
	We first show that the error of the HDG solution $\bm\sigma_h$ in the norm $\|\cdot\|_{\bm{\Sigma}_h}$ can be bounded above by the interpolation error and the $L^2$ error $\|\bm{u}-\bm{u}_h\|_0$.

	\begin{lemma}\label{lemma53} Let $(\bm r,\bm u,p)$ and $\bm{\sigma}_h=(\bm{r}_h ,\bm{u}_h ,\widehat{\bm{u}}_h ,p_h ,\widehat{p}_h )\in \bm{\Sigma}_h$ be the solutions of \eqref{mixed} and \eqref{Bh_HDG}, respectively. Then we have the following estimate
		\begin{align}
			\|\bm{\sigma}-\bm{\sigma}_h\|_{\bm{\Sigma}_h}\lesssim  h^{t} {\big(
				(1+kh){\|\bm r\|_t}+
				({k}+1){\|\bm{u}\|_t}+({k}+1){\|p\|_{t+1}}\big)}+{k}\|\bm{u}-\bm{u}_h\|_0,
			\label{error1}
		\end{align}
		where  $\bm{\sigma}=\big(\bm r,\bm u,\bm u|_{\mathcal{F}_h}, p,p|_{\mathcal{F}_h}\big)$.
	\end{lemma}
	\begin{proof}
		It follows from the discrete inf-sup condition in \Cref{Th45}, \eqref{eq:DefBh}, the orthogonality in \Cref{lem:Orth},  and \eqref{norm-lambda}  that
		\begin{align*}
			\|\bm{\sigma}_h-\bm{\mathcal{I}}_h\bm{\sigma}\|_{\bm{\Sigma}_h}
			&\lesssim  \sup_{\bm 0\neq \bm{\tau}_h\in\bm{\Sigma}_h}\frac{\mathcal B^+_h(\bm{\sigma}_h-\bm{\mathcal{I}}_h\bm{\sigma},\bm{\tau}_h)}{\|\bm{\tau}_h\| _{\bm{\Sigma}_h}} \\
			&=\sup_{\bm 0\neq \bm{\tau}_h\in\bm{\Sigma}_h}\frac{\mathcal{B}_h^{-}(\bm{\sigma}_h-\bm{\mathcal{I}}_h\bm{\sigma},\bm{\tau}_h)-2{k}^2(\bm{u}_h-\bm{\mathcal{P}}^{\rm curl}_{\ell}\bm{u},\bm{v}_h)}{\|\bm{\tau}_h\|_{\bm{\Sigma}_h}}\\
			&= \sup_{\bm 0\neq \bm{\tau}_h\in\bm{\Sigma}_h}\frac{\mathcal{B}_h^{-}(\bm{\sigma}-\bm{\mathcal{I}}_h\bm{\sigma},\bm{\tau}_h)-2{k}^2(\bm{u}_h-\bm{\mathcal{P}}^{\rm curl}_{\ell}\bm{u},\bm{v}_h)}{\|\bm{\tau}_h\|_{\bm{\Sigma}_h}}\\
			&= \sup_{\bm 0\neq \bm{\tau}_h\in\bm{\Sigma}_h}\frac{\mathcal{B}_h^{+}(\bm{\sigma}-\bm{\mathcal{I}}_h\bm{\sigma},\bm{\tau}_h)+2{k}^2(\bm{u}-\bm{u}_h,\bm{v}_h)}{\|\bm{\tau}_h\|_{\bm{\Sigma}_h}}\\
			&\lesssim \sup_{\bm 0\neq \bm{\tau}_h\in\bm{\Sigma}_h}\frac{\mathcal{B}_h^{+}(\bm{\sigma}-\bm{\mathcal{I}}_h\bm{\sigma},\bm{\tau}_h)}{\|\bm{\tau}_h\|_{\bm{\Sigma}_h}}+k\|\bm{u}-\bm{u}_h\|_0,
		\end{align*}
		which together with \Cref{lemma47,l53} and the triangle inequality completes the proof of the lemma.
	\end{proof}

	Finally, we give the error estimates of the proposed HDG method in the following theorem.
	
	\begin{theorem}\label{thm:L^2} Let  {$(\bm{r}, \bm{u},p)$} and $\bm{\sigma}_h=(\bm{r}_h ,\bm{u}_h ,\widehat{\bm{u}}_h ,p_h ,\widehat{p}_h )\in \bm{\Sigma}_h$ be the solutions of \eqref{source} and \eqref{Bh_HDG}, respectively.
		Suppose $(\bm{r},\bm{u},p)\in [H^t(\Omega)]^{3}\times [H^t(\Omega)]^3 \times H^{t+1}(\Omega)$ with $t\in (1/2,\ell]$. Let $s\in (1/2,1]$ and $s^*\in (1,2]$ be the regularity indexes from \Cref{reg}.
		
		{\rm (i)}  There exists a constant $C_0>0$ independent of $k$ and $h$ such that, if $ \mathcal  M_k kh^{s}\le C_0$,  we have
		\begin{subequations}
			\begin{align}
				\|\bm{r}-\bm{r}_h\|_0&\lesssim  h^{t}\big(({k}+1){\|\bm{u}\|_t}+{\|\bm{r}\|_t}+({k}+1){\|p\|_{t+1}}\big),\label{eq:rhEst}\\
				\|\bm{u}-\bm{u}_h\|_0&\lesssim    h^{t}{\|\bm{u}\|_t}+{\mathcal M_k}h^{t+s}\big({\|\bm{u}\|_t}+{\|\bm{r}\|_t}+({k}+1){\|p\|_{t+1}}\big),\label{eq:uhEst}\\
				\|\nabla p-\nabla_h p_h\|_0&\lesssim   h^{t}\big({\|\bm{u}\|_t}+(k+1)^{-1}{\|\bm{r}\|_t}+{\|p\|_{t+1}}\big).\label{eq:phEst}
			\end{align}	
		\end{subequations}
		
		{\rm (ii)} If  in addition $\bm u\in [H^{t+s}(\Omega)]^3$, it holds {  that}
		\begin{align}\label{eq:uhEst2}
			\|\bm{u}-\bm{u}_h\|_0&\lesssim    h^{t+s}\|\bm{u}\|_{t+s}+ {\mathcal M_k}   h^{t+s}\big(({k}+1){\|\bm{u}\|_t}+{\|\bm{r}\|_t}+({k}+1){\|p\|_{t+1}}\big).
		\end{align}
		
		{\rm (iii)} If $\Omega$ is convex and in addition $\bm u\in [H^{t+1}(\Omega)]^3$, there exists a constant $C_0>0$ independent of $k$ and $h$ such that if $kh^s+\mathcal  M_k k^2h^{s^*}\le C_0$,  we have
		\begin{subequations}
			\begin{align}
				\|\bm{r}-\bm{r}_h\|_0&\lesssim  (h^{t}+{\mathcal M_k}k h^{t+1})\big(({k}+1){\|\bm{u}\|_t}+{\|\bm{r}\|_t}+({k}+1){\|p\|_{t+1}}\big),\label{eq:rhEst2}\\
				\|\bm{u}-\bm{u}_h\|_0&\lesssim    h^{t+1}\|\bm{u}\|_{t+1}+{\mathcal M_k}h^{t+1}\big((k+1){\|\bm{u}\|_t}+{\|\bm{r}\|_t}+({k}+1){\|p\|_{t+1}}\big),\label{eq:uhEst3}\\
				\|\nabla p-\nabla_h p_h\|_0&\lesssim    (h^{t}+{\mathcal M_k}kh^{t+1})\big({\|\bm{u}\|_t}+(k+1)^{-1}{\|\bm{r}\|_t}+{\|p\|_{t+1}}\big).\label{eq:phEst2}
			\end{align}	
		\end{subequations}	%
	\end{theorem}
	\begin{proof}  It suffices to prove the $L^2$ error estimates for $\bm{u}_h$ since the other estimates are direct consequences of them and  \Cref{lemma53}. For simplicity, {  we} denote by $ \bm e :=\bm{u}_h-\bm{u}$ and $C_{u,p}:=({k}+1){\|\bm{u}\|_t}+{\|\bm\nabla\times\bm{u}\|_t}+({k}+1){\|p\|_{t+1}}.$ Without loss of generality, we assume that $kh\lesssim 1$. \eqref{error1} implies that the following estimate holds.
		\begin{align}
			\|\bm{\sigma}-\bm{\sigma}_h\|_{\bm{\Sigma}_h}\lesssim &  h^{t}C_{u,p}+{k}\| \bm e \|_0.
			\label{error2}
		\end{align}
		Similar to the proof of \eqref{EP2},  we decompose the square of the $L^2$ error as
		\begin{align}\label{S0}
			\| \bm e \|_0^2=(\bm{\mathcal P}_{\ell}^{\rm curl}\bm u-\bm u, \bm e )+(\bm{u}_h-\bm{\Pi}_{\ell}^{\rm curl,c}\bm u_h, \bm e )
			+(\bm w_h-\bm \Theta, \bm e )+(\bm \Theta, \bm e ),
		\end{align}
		where $\bm w_h:=\bm{\Pi}_{\ell}^{\rm curl,c}\bm u_h-\bm{\mathcal{P}}^{\rm curl}_{\ell}\bm u$ and the $\bm\Theta$ is defined as in \Cref{Hip}. Let $\bm w_h$ be decomposed as \eqref{DHD} in \Cref{Hip}. For any $q_h\in H_0^1(\Omega)\cap M_h$, from \eqref{ororc}, we conclude that $(\bm{u}_h,\nabla q_h)=0$.  Noting that $ \nabla\cdot\bm u=0$, we have
		\eqn{(\bm w_h-\bm\Theta, \bm e )=(\bm{z}_h+\nabla \xi_h-\bm\Theta, \bm e )=(\bm{z}_h-\bm\Theta, \bm e ).}
		Similar to \eqref{et2}, we have
		\begin{align}\label{S1}
			(\bm{u}_h-\bm{\Pi}_{\ell}^{\rm curl,c}\bm u_h, \bm e )+(\bm w_h-\bm \Theta, \bm e )\
			&\lesssim h^{s}\big(\|\bm{\sigma}-\bm{\sigma}_h\|_{\bm{\Sigma}_h}+\|\bm\nabla\times (\bm{\mathcal{P}}_{\ell}^{\rm curl}\bm u-\bm u)\|_0\big)\| \bm e \|_0\notag\\
			&\lesssim h^{s}\big(h^{t}C_{u,p}+{k}\| \bm e \|_0\big)\| \bm e \|_0,
		\end{align}
		where we have used \eqref{error2} and \Cref{lemma4.2} to derive the last inequality.
		
		We introduce the following duality problem:  Find $(\bm r^d,\bm u^d ,p^d)$ such that
		\begin{subequations}\label{eq:duality}
			\begin{align}
				\bm{r}^d-\bm\nabla\times\bm{u}^d&=0&\text{ in }\Omega, \label{dual}\\
				\bm\nabla\times \bm{r}^d-{k}^2\bm u^d+({k}^2+1)\nabla p^d&=\bm\Theta&\text{ in }\Omega,  \\
				\nabla\cdot\bm{u}^d&=0& \text{ in }\Omega,  \\
				\bm{n}\times\bm{u}^d &=\bm{0}& \text{ on }\Gamma,\\
				p^d&=0& \text{ on }\Gamma.
			\end{align}
		\end{subequations}
		By \Cref{reg}, we have the following regularity estimate of problem \eqref{eq:duality}:
		\begin{align} \label{436}
			\|\bm{r}^d\|_{s}+({k}+1)\|\bm{u}^d\|_{s} \lesssim   {\mathcal M_k} \|\bm\Theta\|_0,
		\end{align}
		and
		\eq{\label{regud}
			\|\bm u^d\|_{s^*}\lesssim {\mathcal M_k} \|\bm{\bm\Theta}\|_0,\quad\text{if $\Omega$ is convex}.
		}
		By combining \Cref{lemma4.2} and \eqref{436}--\eqref{regud}, we have
		\eq{\label{S2}
			\|\bm u^d-\bm{\mathcal{P}}_{\ell}^{\rm curl}\bm u^d\|_0
			\lesssim \mathcal M_kD_{k,h}\|\bm\Theta\|_0 ,}
		where
		\eq{\label{Dkh}
			D_{k,h}:=\begin{cases} h^{s^*} &\text{if $\Omega$ is convex,}\\
				(k+1)^{-1}h^s &\text{otherwise.}
			\end{cases}
		}
		It is easy to check that $\bm{\sigma}^d=(\bm{r}^d, \bm{u}^d, \bm{u}^d|_{\mathcal{F}_h}, $ $p^d, p^d|_{\mathcal{F}_h})$ satisfies the following variational formulation
		\begin{eqnarray}
			\mathcal{B}_h^{-}(\bm{\sigma}^d,\bm{\tau})=-(\bm\Theta,\bm{v})\qquad \forall\,  \bm{\tau}\in\bm{\Sigma}.\label{441}
		\end{eqnarray}
		Let $\bm{u}_h^{dP}$ denote the second component of $ \bm{\MP}_h\bm\sigma^d$. Then by \Cref{lemma47,l53},  the regularity estimate \eqref{436},  the orthogonality in \Cref{lem:Orth}, \Cref{EPerr}, \eqref{S2}, and the fact that $p^d=0$, we have

		\begin{align}
			\begin{split}\label{S3}
				(\bm\Theta, \bm e )=&\mathcal{B}_h^{-}(\bm{\sigma}^d,\bm{\sigma}-\bm{\sigma}_h)=\mathcal{B}_h^{-}(\bm{\sigma}-\bm{\sigma}_h,\bm{\sigma}^d)=\mathcal{B}_h^{-}(\bm{\sigma}-\bm{\sigma}_h,\bm{\sigma}^d- \bm{\MP}_h\bm{\sigma}^d) \\
				=&\mathcal{B}_h^{-}(\bm{\sigma}^d- \bm{\MP}_h\bm{\sigma}^d,\bm{\sigma}-\bm{\sigma}_h)\\
				=&\mathcal{B}_h^{+}(\bm{\sigma}^d- \bm{\MP}_h\bm{\sigma}^d,\bm{\sigma}-\bm{\sigma}_h)+2k^2(\bm{u}^d-\bm{u}_h^{dP},\bm{u}-\bm{u}_h) \\
				\lesssim&\mathcal{B}_h^{+}(\bm{\sigma}^d- \bm{\MP}_h\bm{\sigma}^d,\bm{\sigma}-\bm{\mathcal{I}}_h\bm{\sigma})+k^2\|\bm{u}^d-\bm{u}_h^{dP}\|_0\| \bm e \|_0 \\
				\lesssim&   h^{s}(\|\bm{r}^d\|_s+({k}+1)\|\bm{u}^d\|_s)   \|\bm{\sigma}-\bm{\mathcal{I}}_h\bm{\sigma}\|_{\bm{\Sigma}_h} \\
				&+k^2\big(\|\bm u^d-\bm{\mathcal{P}}_{\ell}^{\rm curl}\bm u^d\|_0+h^{2s}(\|\bm{r}^d\|_s+({k}+1)\|\bm{u}^d\|_s)\big)\| \bm e \|_0
				\\
				\lesssim&  \mathcal M_k  \|\bm\Theta\|_0\big( h^{s}  \|\bm{\sigma}-\bm{\mathcal{I}}_h\bm{\sigma}\|_{\bm{\Sigma}_h}+k^2(D_{k,h}+h^{2s})\| \bm e \|_0\big) \\
				\lesssim&  \mathcal M_k  \|\bm\Theta\|_0\big( h^{s+t} C_{u,p}+k^2(D_{k,h}+h^{2s})\| \bm e \|_0\big),
			\end{split}
		\end{align}
		On the other hand, from \Cref{Hip,pic,lemma4.2}, $ \bm w_h=\bm{\Pi}_{\ell}^{\rm curl,c}\bm u_h-\bm u_h+\bm \xi+\bm u-\bm{\mathcal{P}}_{\ell}^{\rm curl}\bm u$, \eqref{error2}, and the triangle inequality, we have
		\begin{align}\label{theta-theta}
			\begin{split}
				\|\bm\Theta\|_0\le&  {\|\bm\Theta-\bm z_h\|_0+\|\bm z_h\|_0
					\lesssim  h^s\|\bm\nabla\times \bm w_h\|_0+\|\bm w_h\|_0 } \\
				\lesssim& h^{s}\big(\| \bm\nabla_h\times  (\bm{\Pi}_{\ell}^{\rm curl,c}\bm u_h-\bm u_h)\|_0+\|\bm\nabla\times \bm \xi\|_0
				+\|\bm\nabla\times (\bm u-\bm{\mathcal{P}}^{\rm curl}_{\ell}\bm u)\|_0\big)\\
				&+\|\bm{\Pi}_{\ell}^{\rm curl,c}\bm u_h-\bm u_h\|_0
				+\|\bm \xi\|_0+\|\bm{\mathcal P}_{\ell}^{\rm curl}\bm u-\bm u\|_0 \\
				\lesssim&   {h^{s}\big( \|\bm\nabla\times(\bm{\mathcal P}_{\ell}^{\rm curl}\bm u-\bm u)\|_0+\|\bm{\sigma}_h-\bm\sigma\|_{\bm\Sigma_h}     \big)
					+\|\bm{\mathcal P}_{\ell}^{\rm curl}\bm u-\bm u\|_0+\| \bm e \|_0}\\
				\lesssim&   {h^{s+t}C_{u,p}
					+\|\bm{\mathcal P}_{\ell}^{\rm curl}\bm u-\bm u\|_0+(1+kh^s)\| \bm e \|_0}.
			\end{split}
		\end{align}
		Combining \eqref{S0}, \eqref{S1},  {\eqref{S3}, \eqref{theta-theta},}  and the Young's inequality gives
		\begin{align*}
			\| \bm e \|_0\lesssim &
			\big(1+\mathcal M_k(k^2D_{k,h}+k^2h^{2s})\big)\|\bm{\mathcal P}_{\ell}^{\rm curl}\bm u-\bm u\|_0\\
			&+\mathcal M_k h^{s+t}(1+kh^s+k^2D_{k,h}+k^2h^{2s})C_{u,p}\\
			&+\big(kh^s+\mathcal M_k(1+kh^s)(k^2D_{k,h}+k^2h^{2s})\big) \| \bm e \|_0.
		\end{align*}
		Therefore, if $kh^s+\mathcal M_k k^2D_{k,h}$ is sufficiently small, it holds
		\begin{align}\label{uuh}
			\|\bm{u}-\bm{u}_h\|_0&\lesssim \|\bm{\mathcal P}_{\ell}^{\rm curl}\bm u-\bm u\|_0+\mathcal M_k h^{s+t}C_{u,p},
		\end{align}
		which implies \eqref{eq:uhEst}, \eqref{eq:uhEst2}, and \eqref{eq:uhEst3}.
		This completes the proof of the theorem.
	\end{proof}
	
	As a consequence of the above theorem, we have the following well-posedness of the proposed HDG method.
	\begin{corollary} Under the conditions of \Cref{thm:L^2}, the HDG scheme \eqref{Bh_HDG} has a unique solution $\bm{\sigma}_h\in\bm{\Sigma}_h $.
	\end{corollary}
	
	In view of \Cref{reg}, \eqref{uuh},  and \Cref{lemma53,lemma4.2}, we can obtain the following error estimates for the linear HDG method  as it is applied on a convex domain.
	\begin{corollary}\label{cor0}
		Suppose $\Omega$ is convex, $\ell=1$,  and  $\nabla\cdot\bm f=0$. Then there exists a constant $C_0>0$ independent of $k$ and $h$ such that if $ kh^s+\mathcal  M_k k^2h^{s^*}\le C_0$,  the following error estimates hold:
		\begin{align*}
			\|\bm{r}_h-\bm{r}\|_0&\lesssim  ({\mathcal M_k}h+{\mathcal M_k} k h^{s^*}+{\mathcal M_k^2}kh^2)  \|\bm{f}\|_0,\\
			\|\bm{u}_h-\bm{u}\|_0&\lesssim ({\mathcal M_k}  h^{s^*}+{\mathcal M}_k^2h^2) \|\bm{f}\|_0.
		\end{align*}
	\end{corollary}

	\section{Numerical experiments}\label{section6} The numerical tests are programmed   using the C++ library named MFEM \cite{MR4189802}. 
	 The solver of GMRES  with block AMG as preconditioner is chosen to solve the linear systems. 
	%
	Let $\Omega=[0,1]^3$. We take the following exact solution $\bm u$ and $p$, and compute the functions $\bm{r}$ and $\bm{f}$ accordingly.
	{
	
		\begin{align*}
			\bm u&=[\sin(\pi y)\sin(\pi z),\sin(\pi z)\sin(\pi x),\sin(\pi x)\sin(\pi y)]^T,\\
			\bm r&=[
			\pi\sin(\pi x)(\cos(\pi y)-\cos(\pi z)),
			\pi\sin(\pi y)(\cos(\pi z)-\cos(\pi x)), \\
			&\quad
			\pi\sin(\pi z)(\cos(\pi x)-\cos(\pi y))]^T,\\
			p&=2\pi\sin(2\pi x)\sin(2\pi y)\sin(2\pi z).
		\end{align*}

	}
	%
	%
	%
	
	The numerical results for $\ell=m =1$ and $\ell=m =2$ are all presented to illustrate the performance of the proposed HDG method. We take $k=0,1,2,4$ and report the errors in \Cref{com1,com2}, respectively. It can be observed that: the convergence orders of $\|\bm r_h-\bm r\|_0$, $\|\bm u_h-\bm u\|_0$ and $\|\nabla_h(p_h-p)\|_0$ can obtain the convergence orders of almost $\ell$, $\ell+1$ and $\ell$ as $h\to 0$, which confirms our theoretical analysis.
	
	{
		
		\begin{table}[!h]	
			\scriptsize \label{com1}
			\caption{\label{tab1}History of convergence for $\ell=m=1$}
			\centering
			
			\begin{tabular}{c|c|c|c|c|c|c|c|c}
				\Xhline{1pt}
				\multirow{2}{*}{$k$} &
				
				\multirow{2}{*}{$\frac{\sqrt{3}}{h}$} &
				
				\multicolumn{2}{c|}{${\|\bm{r}_h-\bm{r}\|_0/\|\bm r\|_0}$} &
				\multicolumn{2}{c|}{${\|\bm{u}_h-\bm{u}\|_0/\|\bm u\|_0}$} &
				
				\multicolumn{2}{c| }{${\|\nabla_h( p_h- p)\|_0/\|\nabla p\|_0}$} &
				\multirow{2}{*}{DOF}\\
				
				\cline{3-8}
				
				&  &Error &Rate  &Error &Rate  &Error &Rate  \\
				\hline
				
				\multirow{5}{*}{$0$}
				&	2	&	1.2990E-01	&		&	1.7590E-01	&		&	6.4780E-01	&		&	3072	\\
				&	4	&	3.7800E-02	&	1.78 	&	4.3740E-02	&	2.01 	&	3.0820E-01	&	1.07 	&	23424	\\
				&	8	&	1.0980E-02	&	1.78 	&	1.1870E-02	&	1.88 	&	1.0900E-01	&	1.50 	&	182784	\\
				&	16	&	3.7210E-03	&	1.56 	&	3.2110E-03	&	1.89 	&	3.2620E-02	&	1.74 	&	1443840	\\
				&	32	&	1.5590E-03	&	1.26 	&	8.2780E-04	&	1.96 	&	8.6930E-03	&	1.91 	&	11476992	\\

				\cline{1-9}
				\multirow{4}{*}{$1$}
				&	2	&	1.3380E-01	&		&	1.7020E-01	&		&	6.5900E-01	&		&	3072	\\
				&	4	&	3.9670E-02	&	1.75 	&	3.9520E-02	&	2.11 	&	3.3500E-01	&	0.98 	&	23424	\\
				&	8	&	1.1630E-02	&	1.77 	&	1.0800E-02	&	1.87 	&	1.4180E-01	&	1.24 	&	182784	\\
				&	16	&	3.8400E-03	&	1.60 	&	3.0830E-03	&	1.81 	&	4.8660E-02	&	1.54 	&	1443840	\\
				&	32	&	1.5750E-03	&	1.29 	&	8.1790E-04	&	1.92 	&	1.3760E-02	&	1.82 	&	11476992	\\

				\cline{1-9}
				\multirow{4}{*}{$2$}
				
				&	2	&	1.4020E-01	&		&	1.7620E-01	&		&	6.6850E-01	&		&	3072	\\
				&	4	&	4.2840E-02	&	1.71 	&	3.7500E-02	&	2.23 	&	3.6720E-01	&	0.86 	&	23424	\\
				&	8	&	1.3140E-02	&	1.70 	&	9.4480E-03	&	1.99 	&	2.0510E-01	&	0.84 	&	182784	\\
				&	16	&	4.2060E-03	&	1.64 	&	2.8070E-03	&	1.75 	&	9.1510E-02	&	1.17 	&	1443840	\\
				&	32	&	1.6270E-03	&	1.37 	&	7.9130E-04	&	1.83 	&	2.9730E-02	&	1.62 	&	11476992	\\

				\cline{1-9}
				\multirow{4}{*}{$4$}
				&	2	&	1.6100E-01	&		&	2.3530E-01	&		&	6.8000E-01	&		&	3072	\\
				&	4	&	5.8490E-02	&	1.46 	&	5.0720E-02	&	2.21 	&	3.8160E-01	&	0.83 	&	23424	\\
				&	8	&	1.8370E-02	&	1.67 	&	1.2960E-02	&	1.97 	&	2.7520E-01	&	0.47 	&	182784	\\
				&	16	&	5.6220E-03	&	1.71 	&	3.0170E-03	&	2.10 	&	1.8230E-01	&	0.59 	&	1443840	\\
				&	32	&	1.8690E-03	&	1.59 	&	7.8750E-04	&	1.94 	&	8.0350E-02	&	1.18 	&	11476992	\\

				\Xhline{1pt}
			\end{tabular}
			
		\end{table}
		\begin{table}[!h]	
			\scriptsize \label{com2}
			\caption{\label{tab2}History of convergence for $\ell=m=2$}
			\centering
			
			\begin{tabular}{c|c|c|c|c|c|c|c|c}
				\Xhline{1pt}
				\multirow{2}{*}{$k$} &
				
				\multirow{2}{*}{$\frac{\sqrt{3}}{h}$} &
				
				\multicolumn{2}{c|}{${\|\bm{r}_h-\bm{r}\|_0/\|\bm r\|_0}$} &
				\multicolumn{2}{c|}{${\|\bm{u}_h-\bm{u}\|_0/\|\bm u\|_0}$} &
				
				\multicolumn{2}{c| }{${\|\nabla_h( p_h- p)\|_0/\|\nabla p\|_0}$} &
				\multirow{2}{*}{DOF}\\
				
				\cline{3-8}
				
				&  &Error &Rate  &Error &Rate  &Error &Rate  \\
				\hline
				
				\multirow{5}{*}{$0$}
				&	2	&	2.9080E-02	&		&	4.4280E-02	&		&	3.6540E-01	&		&	6480	\\
				&	4	&	4.3280E-03	&	2.75 	&	5.1880E-03	&	3.09 	&	7.6440E-02	&	2.26 	&	49728	\\
				&	8	&	6.3910E-04	&	2.76 	&	6.1540E-04	&	3.08 	&	1.1580E-02	&	2.72 	&	389376	\\
				&	16	&	1.1660E-04	&	2.46 	&	7.7450E-05	&	2.99 	&	1.5160E-03	&	2.93 	&	3081216	\\
				&	32	&	2.5900E-05	&	2.17 	&	9.7300E-06	&	2.99 	&	1.9110E-04	&	2.99 	&	24514560	\\
				
				\cline{1-9}
				
				\multirow{5}{*}{$1$}
				&	2	&	2.9600E-02	&		&	3.9680E-02	&		&	3.6470E-01	&		&	6480	\\
				&	4	&	4.4200E-03	&	2.74 	&	4.6010E-03	&	3.11 	&	7.6970E-02	&	2.24 	&	49728	\\
				&	8	&	6.5480E-04	&	2.76 	&	5.8080E-04	&	2.99 	&	1.2400E-02	&	2.63 	&	389376	\\
				&	16	&	1.1780E-04	&	2.48 	&	7.6030E-05	&	2.93 	&	1.7140E-03	&	2.86 	&	3081216	\\
				&	32	&	2.5980E-05	&	2.18 	&	9.6820E-06	&	2.97 	&	2.2040E-04	&	2.96 	&	24514560	\\

				\cline{1-9}
				
				\multirow{5}{*}{$2$}
				
				&	2	&	3.0650E-02	&		&	3.8060E-02	&		&	3.6690E-01	&		&	6480	\\
				&	4	&	4.6780E-03	&	2.71 	&	4.1060E-03	&	3.21 	&	8.0130E-02	&	2.20 	&	49728	\\
				&	8	&	7.0510E-04	&	2.73 	&	5.2080E-04	&	2.98 	&	1.5520E-02	&	2.37 	&	389376	\\
				&	16	&	1.2200E-04	&	2.53 	&	7.2450E-05	&	2.85 	&	2.5710E-03	&	2.59 	&	3081216	\\
				&	32	&	2.6230E-05	&	2.22 	&	9.5440E-06	&	2.92 	&	3.5810E-04	&	2.84 	&	24514560	\\

				\cline{1-9}
				
				\multirow{5}{*}{$4$}
				&	2	&	3.2980E-02	&		&	4.2250E-02	&		&	3.7000E-01	&		&	6480	\\
				&	4	&	5.1700E-03	&	2.67 	&	4.1080E-03	&	3.36 	&	8.4460E-02	&	2.13 	&	49728	\\
				&	8	&	8.3680E-04	&	2.63 	&	4.6580E-04	&	3.14 	&	2.1970E-02	&	1.94 	&	389376	\\
				&	16	&	1.3940E-04	&	2.59 	&	6.4140E-05	&	2.86 	&	5.4500E-03	&	2.01 	&	3081216	\\
				&	32	&	2.7370E-05	&	2.35 	&	9.0780E-06	&	2.82 	&	9.5950E-04	&	2.51 	&	24514560	\\

				\Xhline{1pt}
			\end{tabular}
			
		\end{table}
	}


\begin{thebibliography}{10}
	
	\bibitem{MR1609607}
	{\sc A.~Alonso and A.~Valli}, {\em An optimal domain decomposition
		preconditioner for low-frequency time-harmonic {M}axwell equations}, Math.
	Comp., 68 (1999), pp.~607--631.
	
	\bibitem{MR1626990}
	{\sc C.~Amrouche, C.~Bernardi, M.~Dauge, and V.~Girault}, {\em Vector
		potentials in three-dimensional non-smooth domains}, Math. Methods Appl.
	Sci., 21 (1998), pp.~823--864.
	
	\bibitem{MR4189802}
	{\sc R.~Anderson, J.~Andrej, A.~Barker, and et~al.}, {\em M{FEM}: {A} modular
		finite element methods library}, Comput. Math. Appl., 81 (2021), pp.~42--74.
	
	\bibitem{MR744924}
	{\sc A.~Bendali}, {\em Numerical analysis of the exterior boundary value
		problem for the time-harmonic {M}axwell equations by a boundary finite
		element method. {II}. {T}he discrete problem}, Math. Comp., 43 (1984),
	pp.~47--68.
	
	\bibitem{MR3207533}
	{\sc Y.~Boubendir and C.~Turc}, {\em Well-conditioned boundary integral
		equation formulations for the solution of high-frequency electromagnetic
		scattering problems}, Comput. Math. Appl., 67 (2014), pp.~1772--1805.
	
	\bibitem{Brenner2006DivF}
	{\sc S.~Brenner, F.~Li, and L.-Y. Sung}, {\em A locally divergence-free
		nonconforming finite element method for the time-harmonic {Maxwell}
		equations}, Math. Comp., 76 (2006), pp.~573--595.
	
	\bibitem{Brenner2008IPCurl}
	\leavevmode\vrule height 2pt depth -1.6pt width 23pt, {\em A locally
		divergence-free interior penalty method for the two-dimensional curl-curl
		problems}, SIAM J. Numer. Anal, 46 (2008), pp.~1190--1211.
	
	\bibitem{MR1935806}
	{\sc A.~Buffa, M.~Costabel, and C.~Schwab}, {\em Boundary element methods for
		{M}axwell's equations on non-smooth domains}, Numer. Math., 92 (2002),
	pp.~679--710.
	
	\bibitem{MR2324460}
	{\sc A.~Buffa, P.~Houston, and I.~Perugia}, {\em Discontinuous {G}alerkin
		computation of the {M}axwell eigenvalues on simplicial meshes}, J. Comput.
	Appl. Math., 204 (2007), pp.~317--333.
	
	\bibitem{MR2263045}
	{\sc A.~Buffa and I.~Perugia}, {\em Discontinuous {G}alerkin approximation of
		the {M}axwell eigenproblem}, SIAM J. Numer. Anal., 44 (2006), pp.~2198--2226.
	
	\bibitem{MR2511729}
	{\sc A.~Buffa, I.~Perugia, and T.~Warburton}, {\em The mortar-discontinuous
		{G}alerkin method for the 2{D} {M}axwell eigenproblem}, J. Sci. Comput., 40
	(2009), pp.~86--114.
	
	\bibitem{be07}
	{\sc E.~Burman and A.~Ern}, {\em Continuous interior penalty $hp$-finite
		element methods for advection and advection-diffusion equations}, Math.
	Comp., 259 (2007), pp.~1119--1140.
	
	\bibitem{MR3044180}
	{\sc A.~Cesmelioglu, B.~Cockburn, N.~C. Nguyen, and J.~Peraire}, {\em Analysis
		of {HDG} methods for {O}seen equations}, J. Sci. Comput., 55 (2013),
	pp.~392--431.
	
	\bibitem{MR2833489}
	{\sc B.~Chabaud and B.~Cockburn}, {\em Uniform-in-time superconvergence of
		{HDG} methods for the heat equation}, Math. Comp., 81 (2012), pp.~107--129.
	
	\bibitem{MR4423464}
	{\sc T.~Chaumont-Frelet and P.~Vega}, {\em Frequency-explicit approximability
		estimates for time-harmonic {M}axwell's equations}, Calcolo, 59 (2022),
	pp.~Paper No. 22, 15.
	
	\bibitem{MR4036984}
	{\sc G.~Chen, B.~Cockburn, J.~Singler, and Y.~Zhang}, {\em Superconvergent
		interpolatory {HDG} methods for reaction diffusion equations {I}: {A}n {${\rm
				HDG}_k$} method}, J. Sci. Comput., 81 (2019), pp.~2188--2212.
	
	\bibitem{MR4092295}
	{\sc G.~Chen and J.~Cui}, {\em On the error estimates of a hybridizable
		discontinuous {G}alerkin method for second-order elliptic problem with
		discontinuous coefficients}, IMA J. Numer. Anal., 40 (2020), pp.~1577--1600.
	
	\bibitem{Chen-2018-01}
	{\sc G.~Chen, J.~Cui, and L.~Xu}, {\em Analysis of a hybridizable discontinuous
		{G}alerkin method for the {M}axwell operator}, ESAIM Math. Model. Numer.
	Anal., 53 (2019), pp.~301--324.
	
	\bibitem{MR3813574}
	{\sc G.~Chen, W.~Hu, J.~Shen, J.~R. Singler, Y.~Zhang, and X.~Zheng}, {\em An
		{HDG} method for distributed control of convection diffusion {PDE}s}, J.
	Comput. Appl. Math., 343 (2018), pp.~643--661.
	
	\bibitem{MR3954449}
	{\sc G.~Chen, P.~Monk, and Y.~Zhang}, {\em An {HDG} method for the
		time-dependent drift-diffusion model of semiconductor devices}, J. Sci.
	Comput., 80 (2019), pp.~420--443.
	
	\bibitem{MR3771897}
	{\sc H.~Chen, W.~Qiu, and K.~Shi}, {\em A priori and computable a posteriori
		error estimates for an {HDG} method for the coercive {M}axwell equations},
	Comput. Methods Appl. Mech. Engrg., 333 (2018), pp.~287--310.
	
	\bibitem{MR3608330}
	{\sc H.~Chen, W.~Qiu, K.~Shi, and M.~Solano}, {\em A superconvergent {HDG}
		method for the {M}axwell equations}, J. Sci. Comput., 70 (2017),
	pp.~1010--1029.
	
	\bibitem{Clark1967}
	{\sc C.~{Clark}}, {\em {The Asymptotic Distribution of Eigenvalues and
			Eigenfunctions for Elliptic Boundary Value Problems}}, SIAM Review, 9 (1967),
	pp.~627--646.
	
	\bibitem{MR3626530}
	{\sc B.~Cockburn, G.~Fu, and F.~J. Sayas}, {\em Superconvergence by
		{$M$}-decompositions. {P}art {I}: {G}eneral theory for {HDG} methods for
		diffusion}, Math. Comp., 86 (2017), pp.~1609--1641.
	
	\bibitem{MR2629996}
	{\sc B.~Cockburn, J.~Gopalakrishnan, and F.-J. Sayas}, {\em A projection-based
		error analysis of {HDG} methods}, Math. Comp., 79 (2010), pp.~1351--1367.
	
	\bibitem{MR1672271}
	{\sc M.~Costabel and M.~Dauge}, {\em Maxwell and {L}am\'{e} eigenvalues on
		polyhedra}, Math. Methods Appl. Sci., 22 (1999), pp.~243--258.
	
	\bibitem{MR1753704}
	\leavevmode\vrule height 2pt depth -1.6pt width 23pt, {\em Singularities of
		electromagnetic fields in polyhedral domains}, Arch. Ration. Mech. Anal., 151
	(2000), pp.~221--276.
	
	\bibitem{MR3168286}
	{\sc J.~Cui and W.~Zhang}, {\em An analysis of {HDG} methods for the
		{H}elmholtz equation}, IMA J. Numer. Anal., 34 (2014), pp.~279--295.
	
	\bibitem{MR961439}
	{\sc M.~Dauge}, {\em Elliptic boundary value problems on corner domains},
	vol.~1341 of Lecture Notes in Mathematics, Springer-Verlag, Berlin, 1988.
	\newblock Smoothness and asymptotics of solutions.
	
	\bibitem{MR4081917}
	{\sc S.~Du and F.-J. Sayas}, {\em New analytical tools for {HDG} in elasticity,
		with applications to elastodynamics}, Math. Comp., 89 (2020), pp.~1745--1782.
	
	\bibitem{MR3267105}
	{\sc M.~El~Bouajaji, X.~Antoine, and C.~Geuzaine}, {\em Approximate local
		magnetic-to-electric surface operators for time-harmonic {M}axwell's
		equations}, J. Comput. Phys., 279 (2014), pp.~241--260.
	
	\bibitem{MR3518362}
	{\sc X.~Feng, P.~Lu, and X.~Xu}, {\em A hybridizable discontinuous {G}alerkin
		method for the time-harmonic {M}axwell equations with high wave number},
	Comput. Methods Appl. Math., 16 (2016), pp.~429--445.
	
	\bibitem{MR3265182}
	{\sc X.~Feng and H.~Wu}, {\em An absolutely stable discontinuous {G}alerkin
		method for the indefinite time-harmonic {M}axwell equations with large wave
		number}, SIAM J. Numer. Anal., 52 (2014), pp.~2356--2380.
	
	\bibitem{Maxwell-1997}
	{\sc P.~Fernandes and G.~Gilardi}, {\em Magnetostatic and electrostatic
		problems in inhomogeneous anisotropic media with irregular boundary and mixed
		boundary conditions}, Math. Models Methods Appl. Sci., 7 (1997),
	pp.~957--991.
	
	\bibitem{MR3342199}
	{\sc G.~Fu, W.~Qiu, and W.~Zhang}, {\em An analysis of {HDG} methods for
		convection-dominated diffusion problems}, ESAIM Math. Model. Numer. Anal., 49
	(2015), pp.~225--256.
	
	\bibitem{MR1933811}
	{\sc J.~Gopalakrishnan and J.~E. Pasciak}, {\em Overlapping {S}chwarz
		preconditioners for indefinite time harmonic {M}axwell equations}, Math.
	Comp., 72 (2003), pp.~1--15.
	
	\bibitem{MR3874791}
	{\sc J.~Gopalakrishnan, M.~Solano, and F.~Vargas}, {\em Dispersion analysis of
		{HDG} methods}, J. Sci. Comput., 77 (2018), pp.~1703--1735.
	
	\bibitem{MR2310392}
	{\sc C.~Greif and D.~Sch\"{o}tzau}, {\em Preconditioners for the discretized
		time-harmonic {M}axwell equations in mixed form}, Numer. Linear Algebra
	Appl., 14 (2007), pp.~281--297.
	
	\bibitem{MR3027462}
	{\sc B.~Helffer}, {\em Spectral theory and its applications}, vol.~139 of
	Cambridge Studies in Advanced Mathematics, Cambridge University Press,
	Cambridge, 2013.
	
	\bibitem{MR2009375}
	{\sc R.~Hiptmair}, {\em Finite elements in computational electromagnetism},
	Acta Numer., 11 (2002), pp.~237--339.
	
	\bibitem{Maxwell-IBC}
	{\sc R.~Hiptmair, A.~Moiola, and I.~Perugia}, {\em Stability results for the
		time-harmonic {M}axwell equations with impedance boundary conditions}, Math.
	Models Methods Appl. Sci., 21 (2011), pp.~2263--2287.
	
	\bibitem{MR2983024}
	\leavevmode\vrule height 2pt depth -1.6pt width 23pt, {\em Error analysis of
		{T}refftz-discontinuous {G}alerkin methods for the time-harmonic {M}axwell
		equations}, Math. Comp., 82 (2013), pp.~247--268.
	
	\bibitem{MR2194528}
	{\sc P.~Houston, I.~Perugia, A.~Schneebeli, and D.~Sch\"{o}tzau}, {\em Interior
		penalty method for the indefinite time-harmonic {M}axwell equations}, Numer.
	Math., 100 (2005), pp.~485--518.
	
	\bibitem{MR2051073}
	{\sc P.~Houston, I.~Perugia, and D.~Sch\"{o}tzau}, {\em Mixed discontinuous
		{G}alerkin approximation of the {M}axwell operator}, SIAM J. Numer. Anal., 42
	(2004), pp.~434--459.
	
	\bibitem{MR2300291}
	{\sc O.~A. Karakashian and F.~Pascal}, {\em Convergence of adaptive
		discontinuous {G}alerkin approximations of second-order elliptic problems},
	SIAM J. Numer. Anal., 45 (2007), pp.~641--665.
	
	\bibitem{MR3626528}
	{\sc P.~Lu, H.~Chen, and W.~Qiu}, {\em An absolutely stable {$hp$}-{HDG} method
		for the time-harmonic {M}axwell equations with high wave number}, Math.
	Comp., 86 (2017), pp.~1553--1577.
	
	\bibitem{MR2059447}
	{\sc P.~Monk}, {\em Finite element methods for {M}axwell's equations},
	Numerical Mathematics and Scientific Computation, Oxford University Press,
	New York, 2003.
	
	\bibitem{MR592160}
	{\sc J.-C. N\'{e}d\'{e}lec}, {\em Mixed finite elements in {${\bf R}\sp{3}$}},
	Numer. Math., 35 (1980), pp.~315--341.
	
	\bibitem{MR864305}
	\leavevmode\vrule height 2pt depth -1.6pt width 23pt, {\em A new family of
		mixed finite elements in {${\bf R}^3$}}, Numer. Math., 50 (1986), pp.~57--81.
	
	\bibitem{MR1822275}
	{\sc J.-C. N\'{e}d\'{e}lec}, {\em Acoustic and electromagnetic equations},
	vol.~144 of Applied Mathematical Sciences, Springer-Verlag, New York, 2001.
	\newblock Integral representations for harmonic problems.
	
	\bibitem{MR2822937}
	{\sc N.~C. Nguyen, J.~Peraire, and B.~Cockburn}, {\em Hybridizable
		discontinuous {G}alerkin methods for the time-harmonic {M}axwell's
		equations}, J. Comput. Phys., 230 (2011), pp.~7151--7175.
	
	\bibitem{oswald93}
	{\sc P.~Oswald}, {\em On a {BPX}-preconditioner for {P}1 elements}, Computing,
	51 (1993), pp.~125--133.
	
	\bibitem{MR1972732}
	{\sc I.~Perugia and D.~Sch\"{o}tzau}, {\em The {$hp$}-local discontinuous
		{G}alerkin method for low-frequency time-harmonic {M}axwell equations}, Math.
	Comp., 72 (2003), pp.~1179--1214.
	
	\bibitem{MR1929626}
	{\sc I.~Perugia, D.~Sch\"{o}tzau, and P.~Monk}, {\em Stabilized interior
		penalty methods for the time-harmonic {M}axwell equations}, Comput. Methods
	Appl. Mech. Engrg., 191 (2002), pp.~4675--4697.
	
	\bibitem{MR2671288}
	{\sc D.~S\'{a}rm\'{a}ny, F.~Izs\'{a}k, and J.~J.~W. van~der Vegt}, {\em Optimal
		penalty parameters for symmetric discontinuous {G}alerkin discretisations of
		the time-harmonic {M}axwell equations}, J. Sci. Comput., 44 (2010),
	pp.~219--254.
	
	\bibitem{MR2902419}
	{\sc P.~Tsuji, B.~Engquist, and L.~Ying}, {\em A sweeping preconditioner for
		time-harmonic {M}axwell's equations with finite elements}, J. Comput. Phys.,
	231 (2012), pp.~3770--3783.
	
	\bibitem{MR2220916}
	{\sc T.~Warburton and M.~Embree}, {\em The role of the penalty in the local
		discontinuous {G}alerkin method for {M}axwell's eigenvalue problem}, Comput.
	Methods Appl. Mech. Engrg., 195 (2006), pp.~3205--3223.
	
	\bibitem{Zhang2015}
	{\sc Z.~Zhang}, {\em How many numerical eigenvalues can we trust?}, J. Sci.
	Comput., 65 (2015), pp.~455--466.
	
	\bibitem{MR2869030}
	{\sc L.~Zhong, L.~Chen, S.~Shu, G.~Wittum, and J.~Xu}, {\em Convergence and
		optimality of adaptive edge finite element methods for time-harmonic
		{M}axwell equations}, Math. Comp., 81 (2012), pp.~623--642.
	
	\bibitem{MR2536902}
	{\sc L.~Zhong, S.~Shu, G.~Wittum, and J.~Xu}, {\em Optimal error estimates for
		{N}edelec edge elements for time-harmonic {M}axwell's equations}, J. Comput.
	Math., 27 (2009), pp.~563--572.
	
	\bibitem{zhu2013preasymptotic}
	{\sc L.~Zhu and H.~Wu}, {\em Preasymptotic error analysis of {CIP-FEM} and
		{FEM} for {H}elmholtz equation with high wave number. part {II}: hp version},
	SIAM J. Numer. Anal., 51 (2013), pp.~1828--1852.
	
\end{thebibliography}
\end{document}